\documentclass[11pt, a4paper, amsart, fleqn]{article}

\usepackage{latexsym}
\usepackage{amscd}
\usepackage{theorem}
\usepackage{pifont}
\usepackage{amsfonts}
\usepackage{xspace}
\usepackage{amssymb}
\usepackage{fancyhdr}
\usepackage{mathrsfs}
\usepackage{amsmath}%
\usepackage{graphics}%
\usepackage{graphicx}%
\usepackage{euscript}%
\usepackage{psfrag}%
\usepackage{empheq}%
\usepackage{upgreek}
\usepackage{eufrak}
\usepackage{mathtools}

\DeclareMathAlphabet{\mathpzc}{OT1}{pzc}{m}{it}

{\theorembodyfont{\slshape} \newtheorem{theorem}{Theorem}[section]}
{\theorembodyfont{\slshape} \newtheorem{lemma}[theorem]{Lemma}}
{\theorembodyfont{\slshape} \newtheorem{proposition}[theorem]{Proposition}}
{\theorembodyfont{\slshape} \newtheorem{corollary}[theorem]{Corollary}}
{\theorembodyfont{\slshape} \newtheorem{definition}[theorem]{Definition}}
{\theorembodyfont{\slshape} \newtheorem{remark}[theorem]{Remark}}

\newenvironment{proof}{\noindent\textbf{Proof:\ }}{$\hfill{\bullet}$}
\newcommand{\RomanNumeralCaps}[1]
    {\MakeUppercase{\romannumeral #1}}

\numberwithin{equation}{section}

\title{\textsc{A Laplacian on the Full Shift Space}}                

\author{Shrihari Sridharan \\ {\tt shrihari@iisertvm.ac.in}  \bigskip \\ Sharvari Neetin Tikekar \\ {\tt sharvai.tikekar14@iisertvm.ac.in}
 \bigskip \\ 
\small Indian Institute of Science Education and Research \\ \small Thiruvananthapuram(IISER-TVM), India.}     

\date{}

\begin{document}

\maketitle
\thispagestyle{empty}

\begin{abstract}
\noindent
In this paper, we consider the one-sided shift space on finitely many symbols and extend the theory of what is known as rough analysis. We define difference operators on an increasing sequence of subsets of the shift space that would eventually render the Laplacian on the space of real-valued continuous functions on the shift space. We then define the Green's function and the Green's operator that come in handy to solve the analogue to the Dirichlet boundary value problem on the shift space. 
\end{abstract} 
\bigskip \bigskip

\begin{tabular}{l c l}
\textbf{Keywords} & : & Symbolic space \\ & & Laplace operator \\ & & Green's function and Green's operator \\
\textbf{AMS Subject Classifications} & : & 37B10, 31E05, 39A70, 47B60.
\end{tabular}
\bigskip 
\bigskip

\maketitle

\section{Introduction}
\label{introduction}

\noindent
Symbolic dynamics is a relatively new and popular branch of dynamical systems. It is considered as an effective tool in the study of general dynamical systems. The original dynamical system is discretised by equipartitioning the phase space into finitely many subsets, each represented by a different symbol. The trajectory of a point is then observed by tracking the symbols, corresponding to the sets in the partition the point visits, at a given time. This process generates an infinite sequence over finitely many symbols, determined by the partition. The space of all such sequences obtained for a particular labeling of the partition, is known as the symbolic space. Each sequence in the space is a symbolic trajectory corresponding to the dynamical trajectory of a point in the original system.
\medskip

\noindent
The first successful attempt to apply the techniques of symbolic dynamics was made by Hadamard in 1898, to investigate geodesic flows on negatively curved surfaces, in \cite{hadamard}. Forty years later, in \cite{MH}, the term symbolic dynamics was formally proposed by Morse and Hedlund. This foundational work marked the beginning of the study of symbolic dynamics, in its own right. Morse and Hedlund in \cite{MH} analysed the dynamics on the symbolic space in its independent abstract dynamical setting.  Having assigned a metric to the symbolic space, they established that the space is perfect, compact and totally disconnected. They further studied the transitivity and recurrence properties of the dynamics on the space, which are fundamental to any kind of dynamical systems. In 1940, Shannon employed these spaces to model the data and information channels in the theory of communication, in \cite{shannon}. Since then, the branch has found a wide range of applications in ergodic theory, complex dynamics, topological dynamics, number theory, information theory \emph{etc}. Interested readers may refer to \cite{bks, bmn, kitchens, LM} for a detailed literature on the study of symbolic dynamics and its various applications.   
\medskip

\noindent
In this work, we focus on studying the symbolic space on its own merit. It is well known among the dynamicists that the symbolic space can be used to model many naturally occurring non-smooth objects like fractals. The Sierpi\'nski gasket is a popular model of a fractal, that one obtains through an iterated function scheme. Investigations into the analytical study of the Sierpi\'nski gasket through the probabilistic approach were made by Goldstein \cite{goldstein}, Kusuoka \cite{kusuoka} and Barlow and Perkins \cite{barperk}, where the authors constructed a Laplacian on the Sierpi\'nski gasket as a Brownian motion. Kigami formulated a more direct Laplacian on the Sierpi\'nski gasket in \cite{kigami89} and later generalized it for a class of post critically finite self-similar sets in \cite{kigamipcf}. We first summarize the method here.
\medskip

\noindent
Let us begin by constructing the one sided full shift space on $N>1$ symbols. For the symbol set, $S := \{ 1, 2, \cdots, N \}$, consider the space of one sided sequences as,
\[ \Sigma_{N}^{+}\ \ :=\ \ S^{\mathbb{N}}\ \ =\ \ \big\{ x = (x_{1}\, x_{2}\, \cdots)\ :\ x_{i} \in S \big\}. \] 
The \emph{shift operator} $\sigma : \Sigma_{N}^{+} \longrightarrow \Sigma_{N}^{+}$ given by $\sigma (\, (x_{1}\, x_{2}\, x_{3}\cdots)\, ) = ( x_{2}\, x_{3} \cdots)$ is an $N$-to-$1$ continuous transformation. The inverse branches of $\sigma$ are given by, $\sigma_{l} : \Sigma_{N}^{+} \longrightarrow \Sigma_{N}^{+}$ for $l \in S$, which are defined as, $\sigma_{l}\, (x_{1}\, x_{2}\, \cdots\, ) = (l\, x_{1}\, x_{2}\, \cdots\, )$. The pair $(\Sigma_{N}^{+},\, \sigma)$ is known as the \emph{one sided full shift space}. 
\medskip

\noindent
Let $V$ be a set. Let $\ell(V) := \left\lbrace u \, | \, u : V \longrightarrow \mathbb{R}  \right\rbrace$ be the set of all real valued functions on $V$ and  $\mathcal{C}(V)$ denote the set of all real valued continuous functions on $V$. If $V$ is a finite set, then the standard \emph{inner product} on $\ell(V)$, denoted by $\left\langle u,v \right\rangle$, is defined as, 
\[ \left\langle u,\; v \right\rangle\ \ :=\ \ \sum\limits_{p\, \in\, V} u(p) v(p). \] 
For any $q \in V$, its \emph{characteristic function} is defined as,
\[ \chi_{q} (p) =  \begin{cases}
1 & \text{ if}\ p =  q, \\ 
0 & \text{otherwise.}
\end{cases} \] 

\noindent
Let $(K,\; S,\; \{ F_{l} \}_{l\, \in\, S} )$ be a self-similar structure, where $K$ is a compact metrizable topological space, $F_{l} : K \longrightarrow K$ is a continuous injection for each $l \in S$ and there exists a continuous surjection $\pi : \Sigma_{N}^{+} \longrightarrow K$ such that $F_{l} \circ \pi \, = \, \pi \circ \sigma_{l}$ for each $l \in S$. $K$ is called a \emph{post critically finite set} (p.c.f., for short) if the set $P = \bigcup\limits_{n \ge 1} \sigma^{n}  \, \Big( \pi^{-1} \big( \bigcup\limits_{ l \neq j} (F_{l}(K) \cap F_{j}(K))  \big) \Big)$ is finite. Kigami in \cite{kigaminetwork} constructed a compatible sequence $(V_{m}, \, H_{m}) $ of resistance networks (see \cite{doylesnell} for probabilistic study of concepts of electrical networks) on $K$, where for each $m \ge 0$, $V_{m}$ is a finite set with $V_{m} \subset V_{m+1} \subset K $  and $ H_{m}: \ell(V_{m}) \longrightarrow \ell(V_{m})$ is a non-positive definite symmetric linear operator known as the Laplacian. Each $H_{m}$ induces a natural Dirichlet form $\mathcal{E}_{H_{m}}$ on $\ell (V_{m})$. Moreover, the set $V_{*} = \bigcup\limits_{m\, \ge\, 0} V_{m} $ is dense in $K$. Interested readers may refer to \cite{kigamipcf, kigaminetwork, kigamifrac} for definitions and fundamental properties of the Dirichlet forms and difference operators. For the compatibility of the sequence $\{(V_{m},\, H_{m}) \}_{m\, \ge\, 0}$, the following must hold.
\begin{equation}
\label{compatibility_energy}
\mathcal{E}_{H_{m}}(u,u) \ = \ \min \left\lbrace \,\mathcal{E}_{H_{m+1}}(v,v)\, : \, v \in \ell(V_{m+1}), \, v|_{V_{m}} = u \, \right\rbrace \ \ \text{for all}\ \  m \ge 0.
\end{equation}
The `energy' or `resistance form' on $\ell(V_{*})$ is then defined as, 
\[ \mathcal{E}(u,\, v)\ \ =\ \ \lim\limits_{m\, \to\, \infty} \mathcal{E}_{H_{m}} (u|_{V_{m}},\, v|_{V_{m}}),\ \ \ \text{for}\ \ u, v \in \ell(V_{*}), \] 
whenever the limit is finite. In the study of analysis on the general framework of resistance networks, the set $V_{*}$, which is merely a countable set, does not have any topology. To overcome this difficulty, Kigami in \cite{kigamimetric, kigaminetwork} defined and studied a metric on the resistance networks, popularly known as \emph{effective resistance}, given by,
\[R(p,q) \ := \ [\, \min \left\lbrace \mathcal{E}(u, u)\, : \, u(p)=1,\, u(q)=0    \right\rbrace \,] ^{-1}, \ \ \text{for} \ \ p , q \in V_{*}. \]
If $(\Omega, R)$ is the completion of $(V_{*}, R)$, then we have a Laplacian on $\Omega$ associated to the quadratic form $\mathcal{E}$. The important fact to note here is, $\Omega$ can be identified with the original space $K$, if and only if $(\Omega, R)$ is bounded, see \cite{kigamifrac}. In case of p.c.f. self-similar set $K$, this metric $R$ is compatible with the original metric on $K$. Therefore, the Laplacian on $K$ can be directly defined as the renormalized limit of the difference operators $H_{m}$. The effective resistance plays a crucial role in the theory of Laplacians and Dirichlet forms, see \cite{kigamifrac} for a comprehensive study on the topic.
\medskip

\noindent
Further fundamental properties of this Laplacian on the Sierpi\'nski gasket, like the Dirichlet and Neumann problems of the Poisson equation, complete Dirichlet spectrum, heat and wave equations were examined by Kigami \cite{kigami89}, Shima \cite{shima}, Fukushima and Shima \cite{fuku_shima}, Dalrymple, Strichartz and Vinson \cite{dsv} \textit{etc}. Also similar boundary value problems, spectral properties of the Laplacian and various physical phenomena like heat and wave propagation on different and more complex rough spaces have been studied extensively over the last few decades by several mathematicians that include Teplyaev, Alonso-Ruiz, Freiberg, Kesseb\"ohmer \cite{alonsofreiberg, alonsostrichartzetal, guostrichartzetal, hinzteplyaevetal, kessebohmer}. 
\medskip

\noindent
In \cite{denker}, Denker \textit{et al.}, considered the abstract setting of the shift space $\Sigma_{N}^{+}$ independent of its relation with fractals. Rather than constructing a nested sequence of finite sets and difference operators on them as described above, they define `thin' equivalence relations on $\Sigma_{N}^{+}$ and construct Dirichlet forms on the associated quotient spaces of $\Sigma_{N}^{+}$. This Dirichlet form gives rise to the Laplace operator on such quotient spaces. The authors also prove that Kigami's Laplacian on the Sierpi\'nski gasket can be derived as a special case to this theory.
\medskip

\noindent
In the present paper, we consider the same abstract setting of the shift space. We extract a nested sequence of finite subsets $V_{m}$ of $\Sigma_{N}^{+}$ by exploiting the dynamical aspects and the peculiar topology of the symbolic space, and define equivalence relations on each of these subsets as described in section \eqref{sec_settings}. We then define the difference operators in section \eqref{diff operator} and corresponding Dirichlet forms in section \eqref{Dirichlet forms}, on each of these sets. The set $V_{*}$ is unbounded with respect to the analogous effective resistance metric obtained using these Dirichlet forms in this setting, as will be proved in section \eqref{energy}, after introducing concepts like energy and energy minimizers. This establishes that the effective resistance is insufficient in obtaining a Laplacian on the full space $\Sigma_{N}^{+}$. Therefore we resort to the standard metric existing on the shift space, which will be defined in the following section and derive a Laplacian on $\Sigma_{N}^{+} $ as a renormalised limit of the difference operators in section \eqref{the laplacian}.
\medskip

\noindent
The second part of the paper focuses on solving a problem analogous to the Dirichlet boundary value problem, as stated below, through the standard concepts of the Green's function in section \eqref{Green's function}, the Green's operator in section \eqref{Green's operator}.
\begin{theorem} 
\label{maintheorem} 
Let $\mathcal{C} (\Sigma_{N}^{+})$ denote the Banach space of real-valued continuous functions defined on $\Sigma_{N}^{+}$ and $V_{0}$ be the set of fixed points of $\sigma$. For any $f \in \mathcal{C} (\Sigma_{N}^{+})$ and $\zeta \in \ell(V_{0})$, there exists a function $u \in \mathcal{C} (\Sigma_{N}^{+})$ in the domain of the Laplacian that satisfies 
\[ \Delta u\ \ =\ \ f\ \ \ \ \text{subject to}\ \ \ \ u|_{V_{0}}\ \ =\ \ \zeta. \] 
\end{theorem} 
\medskip

\noindent
We conclude the paper by giving a complete solution to the differential equation on this totally disconnected space in section \eqref{The Dirichlet problem}. We aim to investigate the relation between the energy and the Laplacian on $\Sigma_{N}^{+}$ in subsequent papers.  
\medskip

\section{$m$-relations in the full shift space} 
\label{sec_settings} 
The full shift space $\Sigma_{N}^{+}$ introduced in the introduction is equipped with a natural metric defined by, 
\[ d(x, y)\ \ :=\ \ \frac{1}{2^{\, \rho(x, y)}},\ \ \ \text{where}\ \ \rho(x, y)\ :=\ \min \{ i : x_{i} \ne y_{i} \}\ \ \text{with}\ \ \rho(x, x)\ :=\ \infty. \] 
This metric $d$ generates a product topology on $\Sigma_{N}^{+}$, where the symbol set $S$ is considered to have a discrete topology. In fact, for any $0 < \theta < 1$, the metric $d_{\theta} (x, y) := \theta^{\,\rho(x, y)}$ generates the same product topology. Unless otherwise mentioned, we always use the metric $d = d_{\theta}$ with $\theta = \frac{1}{2}$.  In this topology, one may observe that the open sets can be written as a countable union of cylinder sets, which are themselves both closed and open. Thus, the cylinder sets form a basis for the topology on $\Sigma_{N}^{+}$. By \emph{a cylinder set of length $m$}, we mean the set denoted by
\[ [p_{1}\, \cdots\, p_{m}]\ \ :=\ \ \left\lbrace x \in \Sigma_{N}^{+} : x_{1} = p_{1}\, ,\, \cdots\, ,\, x_{m} = p_{m} \right\rbrace, \] 
where we fix the initial $m$ co-ordinates. We wish to draw the attention of the readers to the position of the cylinder sets, which can occur anywhere in general. However, we necessitate the cylinder sets to be placed at the initial co-ordinates, as defined. The reason for the same becomes clear during the course of this section. Under the topology defined, $\Sigma_{N}^{+}$ is a totally disconnected, compact, perfect metric space on which $\sigma$ is a non-invertible, continuous surjection, that has $N$ local inverse branches for any point $x$.   
\medskip 

\noindent 
These cylinder sets form a semi-algebra that would, in turn generate the Borel sigma-algebra on $\Sigma_{N}^{+}$. The \emph{equidistributed Bernoulli measure} $\mu$ on a cylinder set of length $m$ is defined as, 
\begin{equation} 
\label{measure} 
\mu \left( [p_{1}\, \cdots\, p_{m}] \right)\ \ :=\ \ \frac{1}{N^{m}}. 
\end{equation} 

\noindent 
Moreover, the shift space $\Sigma_{N}^{+}$ has a self similar structure. Recall the inverse branches $\sigma_{l} : \Sigma_{N}^{+} \longrightarrow [l]$ for each $l \in S$ as defined in the previous section. Each branch $\sigma_{l}$ maps an element $x$ in $\Sigma_{N}^{+}$ to its preimage (under $\sigma$) that has the symbol $l$ in its first position. $\sigma_{l}$ is a contractive similarity with the contraction ratio being $\frac{1}{2}$. In fact, for distinct $l \in S$, we note that the sets $\sigma_{l} (\Sigma_{N}^{+})$ are mutually disjoint and satisfy, $\Sigma_{N}^{+} = \bigcup\limits_{l\, \in\, S} \sigma_{l}\, (\Sigma_{N}^{+})$. 
\medskip 

\noindent 
We can also generalise the above structure of self-similarity, as follows. Fix $m > 0$ and consider any finite word $w$ of length $|w| = m$ {\it i.e.,} $w = (w_{1}\, w_{2}\, \cdots\, w_{m})$. One can then define the map $\sigma_{w} := \sigma_{w_{1}} \circ \sigma_{w_{2}} \circ \cdots \circ \sigma_{w_{m}} : \Sigma_{N}^{+} \longrightarrow [w_{1}\, w_{2}\, \cdots\, w_{m}]$ which concatenates the finite word $w$ as a prefix to the elements of $\Sigma_{N}^{+}$. Again, we can write the shift space as a disjoint union given by 
\[ \Sigma_{N}^{+} = \bigcup\limits_{ \left\lbrace  w\,:\, |w|\, =\, m \right\rbrace } \sigma_{w} (\Sigma_{N}^{+}). \] 

\noindent
We now understand the shift space $\Sigma_{N}^{+}$ as the limit of an increasing sequence of finite subsets of $\Sigma_{N}^{+}$. For $l \in S$, denote the point $(l\, l\, \cdots\, ) \in \Sigma_{N}^{+}$ by $(\dot{l})$. This is a fixed point of $\sigma$ and of the corresponding map $\sigma_{l}$. Let $V_{0}$ denote the set of all fixed points of $\sigma$, namely,  
\[ V_{0}\ \ :=\ \ \left\lbrace \dot{(1)}\, ,\, \dot{(2)}\, ,\, \cdots\, ,\, \dot{(N)}\, \right\rbrace . \] 
For $m \ge 1$, we define the sets $\left\{ V_{m} \right\}_{m\, \ge\, 1}$ inductively as $V_{m} := \bigcup\limits_{l\, \in\, S} \sigma_{l}\, (V_{m - 1})$. Note that $V_{m}$ is the set of all $m$-th order pre-images of points in $V_{0}$, with cardinality $N^{m + 1}$. Further, the sequence $\{ V_{m} \}$ is increasing. Since $V_{m} = \bigcup\limits_ {\left\lbrace  w\,:\, |w|\, =\, m \right\rbrace} \sigma_{w}\, (V_{0})$, any point $p$ in $V_{m}$ is of the form $p = (p_{1}\, \cdots\, p_{m}\, p_{m + 1}\, p_{m+1}\, \cdots\, )$. We denote this point by $(p_{1}\,  \cdots\, p_{m}\, \dot{p}_{m + 1}\, )$. In particular, for a point $p \in V_{m} \setminus V_{m - 1}$, we have  $p_{m} \ne p_{m + 1}$. 
\medskip 

\noindent 
Define $V_{*} := \bigcup\limits_{m\, \ge\, 0} V_{m}$. Then, $V_{*}$ is a dense subset of $\Sigma_{N}^{+}$ in the standard topology, {\it i.e.,} for any $x = (x_{1}\, x_{2}\, \cdots\, ) \in \Sigma_{N}^{+}$, the sequence of points 
\[ \Big\{ (\dot{x}_{1}) \in V_{0};\ (x_{1}\, \dot{x}_{2}) \in V_{1};\ \cdots;\ (x_{1}\, x_{2}\, \cdots\, x_{m}\, \dot{x}_{m + 1}\, ) \in V_{m}; \cdots \Big\} \] 
converges to $x$. We note that there could be different sequences that approach $x$ in the limit; we have provided only an example of one such to establish density. On each $V_{m}$, we define an equivalence relation to characterize the closest points to a given point in $V_{m}$.
\medskip 

\noindent 
\begin{definition} Any two  points $p = (p_{1}\, p_{2}\, \cdots\, p_{m}\, \dot{p}_{m + 1}\, )$ and $q = (q_{1}\, q_{2}\, \cdots\, q_{m}\, \dot{q}_{m + 1}\, ) $ in $V_{m}$ are said to be $m$-related, denoted by $p \sim_{m} q$, if $p_{i}= q_{i}$ for $1 \le i \le m$.
\end{definition}
\medskip 

\noindent
The definition entails that any two points in $V_{0}$ are $0$-related. The $m$-related points $p, q \in V_{m}$ are obtained by the action of the same $\sigma_{w}$ on $V_{0}$. Any two points in $V_{m}$ are separated by a distance of at least $\frac{1}{2^{m + 1}}$. The $m$-relation, $\sim_{m}$ is an equivalence relation on $V_{m}$. The equivalence class of  $p$ in $V_{m}$ is the set of all $m$-related points of $p$ in $V_{m}$. We denote it by
\[ \left. [p_{1}\,\cdots\, p_{m}] \right|_{V_{m}}\ \ :=\ \ \left\lbrace (p_{1}\, p_{2}\, \cdots\, p_{m}\,\dot{l})\ :\ l \in S \right\rbrace\ \ =\ \ [p_{1}\,\cdots\, p_{m}]\, \cap\, V_{m}. \] 

\noindent 
\begin{remark}
The $m$-relation is clearly reflexive (being an equivalence relation). Therefore, when we say two points are $m$-related, we will only focus on distinct points being $m$-related. We adopt this as a convention, since reflexivity does not play any role in our analysis.
\end{remark}

\noindent 
\begin{remark} 
\label{no.of direct nbrs} 
Let $p = (p_{1}\, p_{2}\, \cdots\, p_{m}\, \dot{p}_{m + 1}) \in V_{m}$. Among all the points in $V_{m}$ other than $p$, those that are $m$-related to $p$ are the closest to $p$ at a distance $\frac{1}{2^{m + 1}}$. We define these points to be the \emph{immediate neighbours} of $p$ in $V_{m}$. We call the set of these immediate neighbours, as the deleted neighbourhood of $p$ in $V_{m}$, denoted by $\mathcal{U}_{p,\, m}$. Observe that there are precisely $N - 1$ immediate neighbours for any $p \in V_{m}$. Let us denote these neighbours by $q^{1},\, q^{2},\,\cdots, \, q^{N-1}$. 
\end{remark}

\noindent
For example, when $N = 3$, the points $(1\,\dot{2}),\, (1\,\dot{3})$ and $(\dot{1})$ in $V_{1}$ are $1$-related to each other. Thus, $\mathcal{U}_{(1\, \dot{2}),\, 1} = \{ (1\,\dot{3}),\, (\dot{1}) \}$. Similarly, the points $(3\, 3\, \dot{1}),\, (3\, 3\,\dot{2})$ and $(\dot{3})$ in $V_{2}$ are $2$-related to each other. Thus, $\mathcal{U}_{(\dot{3}),\, 2} = \{ (3\, 3\, \dot{1}),\, (3\, 3\,\dot{2}) \}$. 
\medskip 

\noindent 
\begin{remark} 
\label{m-related V_m-V_{m-1}} 
Consider $p \in V_{m} \setminus V_{m - 1}$. Among the $N - 1$ neighbours of $p$ accommodated in $\mathcal{U}_{p,\, m}$, we note that one of them denoted by $q^{N-1} = (p_{1}\, p_{2}\, \cdots\, p_{m - 1}\, \dot{p}_{m})$ comes from $V_{m - 1}$. The rest of the neighbours; $N - 2$ of them are collected in the set 
\[ U_{p,\, m}\ \ :=\ \ \left\lbrace  (p_{1}\, p_{2}\, \cdots\, p_{m}\, \dot{l}\, )\ :\ l \ne p_{m} \ \text{and}\ l \ne p_{m + 1} \right\rbrace \ = \ \left\lbrace q^{1},\, q^{2},\,\cdots, \, q^{N-2} \right\rbrace \ \subset\ \ V_{m} \setminus V_{m - 1}. \] 
One can easily verify this for the example when $N = 3$.
\end{remark} 
\medskip 

\noindent 
We make one more interesting observation in the next remark which enables us to track any point in $V_{m}$ from any point in $V_{0}$, through a chain of $k$-related points in the intermediary stages $V_{k}$. Making use of this, we can connect any two points in $V_{m}$ by a chain of $k$-related points from $V_{k}$, from each stage $0 \le k \le m$.
\medskip 

\noindent 
\begin{remark} 
\label{chain connecting pts} 
For $p \in V_{m} \setminus V_{m-1}$, choose $n_{1},\, n_{2},\, \cdots,\, n_{d} \in \mathbb{N}$ such that $1 \le n_{1} < n_{2} < \cdots < n_{d} = m$, which are the only coordinates of $p$ satisfying $p_{n_{i}} \ne p_{n_{i}+1}$. Construct the points 
\[r^{0} = (\dot{p}_{1}) \in V_{0};\ \ r^{n_{i}} = (p_{1}\, p_{2}\, \cdots\, p_{n_{i}}\, \dot{p}_{n_{i}+1}) \in V_{n_{i}} \setminus V_{n_{i} - 1}\ \ \text{and}\ \ r^{n_{d}} = p.   \]
Observe that, any $(\dot{l}) \in V_{0} $ can be connected to $p \in V_{m} \setminus V_{m - 1}$ by means of this chain of distinct points, $r^{0}, r^{n_{1}},\cdots, r^{n_{d}}$, in the sense that, 
\[ (\dot{l}) \sim_{0 } r^{0}\ \text{or} \ (\dot{l}) = r^{0} \ \ \ \ \text{and}\ \ \ \ r^{n_{i - 1}} \sim_{n_{i}} r^{n_{i}}\ \ \text{implying}\ \ r^{n_{i - 1}} \in \mathcal{U}_{r^{n_{i}},\, n_{i}}\ \ \ \text{for}\ 1 \le i \le d. \] 
\end{remark}

\noindent
Due to the peculiar topology on the space $\Sigma_{N}^{+} $, the standard notion of topological boundary becomes irrelevent. Nevertheless, as the method of construction of $\Sigma_{N}^{+}$ begins from $V_{0}$, we define the set $V_{0}$ as the \emph{boundary of $\Sigma_{N}^{+}$}.
\medskip

\section{Difference operators }
\label{diff operator}

\noindent 
In this section, we inductively define a difference operator $H_{m}$ on $\ell (V_{m})$, which gives the total difference between the functional values at a point and its neighbouring points in $V_{m}$. We also provide an easier approach to the operator so defined, by writing its matrix representation, however after defining an order amongst the finitely many points in $V_{m}$. Whenever a point $p$ appears before $q$ in the ordering of $V_{m}$, we denote it by $p \prec q$. Then a matrix for $H_{m}$ is arranged in such a way that, whenever $p \prec q$, the row or column corresponding to the point $p$ appears to the top or to the left of the row or column respectively, corresponding to the point $q$. Let $(H_{m})_{pq}$ denote the entry in the matrix $H_{m}$ corresponding to row $p$ and column $q$. The action of $H_{m}$ on $u \in \ell (V_{m})$ at a point $p \in V_{m}$ is given by, 
\[ H_{m} u (p)\ = \ \sum\limits_{q \,\in \,V_{m}}\, (H_{m})_{pq}\,u(q). \]

\noindent 
For $l \in S$, consider the point $(\dot{l}) \in V_{0}$. Observe that all other points in $V_{0}$ are at an equal distance of $\frac{1}{2}$ from $(\dot{l})$. We then define a difference operator $H_{0}$ on  $\ell (V_{0})$ as,
\begin{equation} 
\label{defn H_0} 
H_{0} u (\dot{l})\ := \ \sum\limits_{(\dot{k})\, \in\, V_{0}} \left( u (\dot{k}) - u (\dot{l}) \right) \ = \ -\, (N - 1)\, u (\dot{l}) + \sum\limits_{\substack{k\, \in\, S \\ k\, \ne\, l}} u (\dot{k}). 
\end{equation}
For $l,\, k \in S$, we define $(\dot{k}) \prec (\dot{l})$ if and only if $k < l$. Then $V_{0}$ can be written in an ascending order of elements as, 
\[ V_{0} \ = \ \left\{ (\dot{1}) \prec (\dot{2}) \prec \cdots \prec (\dot{N}) \right\}.  \]
We write $H_{0}$ as a matrix of order $N$ given by 
\begin{equation*}
\left( H_{0} \right)_{pq}\  = \ 
\begin{cases} 
1 & \text{if}\ q \in \mathcal{U}_{p,\, 0} \\
-\, (N - 1) & \text{if}\ p = q,
\end{cases} 
\end{equation*} 
where $\left( H_{0} \right)_{pq}$ denotes the entry of the matrix $H_{0}$ corresponding to the row for $p \in V_{0}$ and the column for $q \in V_{0}$. 
Thus, 
\[ H_{0}\ \ =\ \ \begin{pmatrix} 
-\, (N - 1) & 1 & 1 & \cdots & 1 \\
1 & -\, (N - 1) & 1 & \cdots & 1 \\
\vdots & & \ddots & & \\
& \vdots & & \ddots & \\
1 & 1 & 1 & \cdots & -\, (N - 1) 
\end{pmatrix}_{N \times N}. \] 
Making use of this matrix representation of $H_{0}$, we observe that for $p \in V_{0}$, we have 
\begin{equation}
\label{altH_0} 
H_{0} u  (p) \ = \ \sum_{q\, \in\, V_{0}} \left( H_{0} \right)_{pq}\, u\, (q). 
\end{equation} 
It is now an easy observation, that the two expressions for $H_{0}$ as in equation \eqref{defn H_0} and \eqref{altH_0} are the same. Having defined a difference operator, namely $H_{0}$ on $\ell (V_{0})$, we now adopt the same philosophy for defining a difference operator $H_{1}$ on $\ell (V_{1})$.  
\medskip 

\noindent 
Let $u \in \ell(V_{1})$ and consider the points $p = (p_{1}\,\dot{p}_{2}) \in V_{1} \setminus V_{0}$ and $(\dot{l}) \in V_{0}$. We define the action of $H_{1}$ on $V_{1} \setminus V_{0}$ and $V_{0}$ separately as,
\begin{eqnarray} 
\label{defn H_1}
H_{1} u(p)\ \ & := &\ \ \sum\limits_{q\, \in\, \mathcal{U}_{p,\, 1}} \left( u(q)\, -\, u(p) \right)\ \ =\ \ -\, (N - 1)\, u(p) + \sum\limits_{q\, \in\, \mathcal{U}_{p,\, 1}} u(q), \\ 
\label{defn H_1a}
H_{1} u (\dot{l})\ \ & := &\ \ -2\, (N - 1)\, u(\dot{l}) + \sum\limits_{\substack{k\, \ne\, l \\ k\, \in\, S}} u(\dot{k})\, + \, \sum\limits_{\substack{q\, \in\, V_{1} \setminus V_{0} \\ q_{1}\, =\, l}} u(q)  \\
& = &\ \ H_{0} u (\dot{l})\, -\, (N - 1)\, u(\dot{l}) \, + \sum\limits_{q\, \in\, \mathcal{U}_{(\dot{l}), 1}} u(q). \nonumber
\end{eqnarray}

\noindent 
By construction, $V_{1}$ contains all the points in $V_{0}$ and its immediate predecessors under $\sigma$. We order the elements of $V_{1}$ as follows: The elements of $V_{0}$ appear first in $V_{1}$, with the prescribed order therein. That is, we define $p \prec q$ for any $p \in V_{0}$ and $q \in V_{1}\setminus V_{0}$. We now define the order on $V_{1} \setminus V_{0}$. For $(\dot{l}) \in V_{0}$ and $i,j \in S$ with $i \ne l$ and $j \ne l$, we define $\sigma_{i}(\dot{l}) \prec \sigma_{j}(\dot{l})$ if and only if $i < j$. And for distinct $(\dot{k}), (\dot{l}) \in V_{0}$ and any $i, j \in S$, we define $\sigma_{i}(\dot{k}) \prec \sigma_{j}(\dot{l})$ if and only if $(\dot{k}) \prec (\dot{l})$. Thus the set $V_{1}$ can be arranged in an ascending order of its points as,
\begin{eqnarray*}
V_{1} \ & = & \ \bigg\{ (\dot{1})\prec \cdots \prec (\dot{N}) \prec (2 \dot{1}) \prec \cdots \prec (N \dot{1}) \prec (1 \dot{2}) \prec \cdots \prec (N \dot{2}) \prec \cdots \\
&  & \hspace{7cm} \cdots \prec (1 \dot{N}) \prec \cdots \prec (N - 1\, \dot{N}) \bigg\}.
\end{eqnarray*}

\noindent 
We now obtain a matrix representation of the operator $H_{1}$. We expect $H_{1}$ to be a square matrix of order $N^{2}$, the cardinality of $V_{1}$. For this purpose, we split $H_{1}$ into $4$ parts as follows:
\begin{equation*}
H_{1}\ \ =\ \ \begin{pmatrix}
T_{1} & J_{1}^{T} \\
J_{1} & X_{1}
\end{pmatrix},
\end{equation*}
where 
\begin{eqnarray*} 
T_{1} : \ell (V_{0}) & \longrightarrow & \ell (V_{0})\ \text{is of order}\ N, \\ 
J_{1} : \ell (V_{0}) & \longrightarrow & \ell (V_{1} \setminus V_{0})\ \text{is of order}\ (N^{2} - N) \times N, \\ 
X_{1} : \ell (V_{1} \setminus V_{0}) & \longrightarrow & \ell (V_{1} \setminus V_{0})\ \text{is of order}\ N^{2} - N. 
\end{eqnarray*}  

\noindent 
Making use of equations \eqref{defn H_1} and \eqref{defn H_1a}, we compute the entries in the relevant matrices $T_{1},\ X_{1}$ and $J_{1}$ as 
\begin{eqnarray*} 
(T_{1})_{pq}\ \ & = &\ \ 
\begin{cases} 
-\, 2 (N - 1) & \text{if}\ p = q, \\ 
1 & \text{otherwise}. \\ 
\end{cases} \\ 
(X_{1})_{pq}\ \ & = & \ \ 
\begin{cases}
-\, (N - 1) & \ \ \text{if}\ p = q, \\
1 & \ \ \text{if}\ q \in \mathcal{U}_{p, 1}, \\ 
0 & \ \ \text{otherwise}. \\
\end{cases} \\ 
(J_{1})_{pq}\ \ & = &\ \ 
\begin{cases}
1 & \qquad \qquad\ \text{if}\ q \in \mathcal{U}_{p, 1}, \\ 
0 & \qquad \qquad\ \text{otherwise}. \\
\end{cases} 
\end{eqnarray*} 

\noindent 
One can observe that the matrices $T_{1},\ J_{1}$ and $X_{1}$ satisfy the following relation: 
\[ T_{1}\ \ =\ \ H_{0} + J_{1}^{T} X_{1}^{-1} J_{1}. \] 

\noindent
We proceed inductively to define the appropriate difference operator $H_{m}$ on $u \in \ell (V_{m})$. If $p \in V_{m}$, then $p \in V_{n} \setminus V_{n-1} $ for some $1 \le n \le m$, or $p \in V_{0}$ (choose $n=0$ in that case). Then,
\begin{eqnarray} 
\label{defn H_m}
H_{m} u (p) & := & -(m-n+1)\,(N-1)\, u(p) + \sum_{i\, = \,n}^{m} \ \sum_{q \,\in\, \mathcal{U}_{p,\, i}} u(q)\\
& = & H_{m - 1} u(p)\, +\, \left[ -\, (N - 1)\, u(p)\, + \sum\limits_{q\, \in\, \mathcal{U}_{p,\, m}} u(q) \right] \nonumber. 
\end{eqnarray} 
In particular, if $p \in V_{m} \setminus V_{m - 1}$, then substituting $n = m$ in \eqref{defn H_m} we obtain,
\begin{equation} 
\label{defn_of_Hm_u}
H_{m} u\,(p)\ := \ -\, (N - 1)\, u(p)\, +\, \sum\limits_{q \, \in  \,\mathcal{U}_{p,\, m}} u(q).
\end{equation} 
To obtain the matrix representation of $H_{m}$, we order the points in $V_{m}$. Recall that any point $p \in V_{m}$ looks like $p = (p_{1}\, p_{2}\, \cdots\, p_{m}\, \dot{p}_{m + 1})$ with $p_{m} = p_{m + 1}$ if $p \in V_{m - 1}$, and $p_{m} \ne p_{m + 1}$ if $p \in V_{m} \setminus V_{m - 1}$. The order of points in $V_{m-1}$ is retained as it is in $V_{m}$. Moreover $V_{m-1}$ appears first in the ordering of $V_{m}$, that is, if $p \in V_{m-1}$ and $q \in V_{m} \setminus V_{m-1}$, then $p \prec q$. Recall that $V_{m} = \bigcup\limits_{i \in S} \sigma_{i}(V_{m-1}) $. Now, for any $p,\,q \in V_{m-1}$, we define $\sigma_{i}(p) \prec \sigma_{j}(p) $ if and only if $i < j$ and for any $i,\,j \in S$, define $ \sigma_{i}(p) \prec \sigma_{j}(q)$ if and only if $ p \prec q$. In summary, the points in $V_{m}$ can be listed in their ascending order as,  
\allowdisplaybreaks
\begin{eqnarray*} 
V_{m} & = & \Bigg\{\ \underbrace{(\dot{1}) \prec \cdots \prec (\dot{N})}_{V_{0}} \prec \underbrace{(2 \dot{1}) \prec \cdots \prec (N \dot{1}) \prec \cdots \prec (1 \dot{N}) \prec \cdots \prec (N - 1\, \dot{N})}_{V_{1} \setminus V_{0}} \prec  \\ 
& & \\
& & \hspace{+2cm} \underbrace{\cdots \prec \cdots \prec \cdots \prec \cdots \prec}_{V_{2} \setminus V_{1},\, \cdots,\, V_{m - 1} \setminus V_{m - 2}} \\
& & \hspace{+5cm} \underbrace{(\underbrace{1\, \cdots\, 1}_{m - 1}\, 2\, \dot{1}) \prec \cdots \prec (\underbrace{N\, \cdots\, N}_{m - 1}\, N - 1\, \dot{N})}_{V_{m} \setminus V_{m - 1}}\ \Bigg\}. 
\end{eqnarray*} 

\noindent 
The matrix representation for the difference operator $H_{m}$ (of order $N^{m + 1}$) defined on $\ell (V_{m})$ is split into four parts, analogous to what we did for $H_{1}$. 
\begin{equation*}
H_{m}\ \ =\ \ \begin{pmatrix} 
T_{m} & J_{m}^{T} \\
J_{m} & X_{m}
\end{pmatrix}
\end{equation*}
where 
\begin{eqnarray*} 
T_{m} : \ell (V_{m - 1}) & \longrightarrow & \ell (V_{m - 1})\ \text{is of order}\ N^{m}, \\ 
J_{m} : \ell (V_{m - 1}) & \longrightarrow & \ell (V_{m} \setminus V_{m - 1})\ \text{is of order}\ (N^{m + 1} - N^{m}) \times N^{m}, \\ 
X_{m} : \ell (V_{m} \setminus V_{m - 1}) & \longrightarrow & \ell (V_{m} \setminus V_{m - 1})\ \text{is of order}\ N^{m + 1} - N^{m}. 
\end{eqnarray*}  

\noindent 
The entries in each of these submatrices are obtained using the definition of $H_{m}$, as given in equation \eqref{defn H_m}.
\begin{eqnarray} 
(T_{m})_{pq}\ \ & = &\ \ 
\begin{cases} 
-(m-n+1)\,(N-1) & \ \ \text{if}\ p = q \in V_{n} \setminus V_{n-1} \ (n < m), \ \text{or} \ V_{0}\ \text(n=0),\\
1 & \ \ \text{if} \ p \in V_{0} \text(n=0) \ \text{or} \ p \in V_{n} \setminus V_{n-1} \ \text{with}\\
&  \qquad  \ q \in \mathcal{U}_{p,\, i} \ \text{for some} \ n \le i < m,\\
0 & \ \ \text{otherwise}. \\ 
\end{cases}  \nonumber\\ 
\label{defn_of_X_m}
(X_{m})_{pq}\ \ & = & \ \ 
\begin{cases}
-\, (N - 1) & \qquad \qquad \quad \ \ \ \text{if}\ p = q, \\
1 & \qquad \qquad \quad \ \ \ \text{if}\ q \in \mathcal{U}_{p,\, m}, \\ 
0 & \qquad \qquad \quad \ \ \ \text{otherwise}. \\
\end{cases} \\ 
(J_{m})_{pq}\ \ & = &\ \ 
\begin{cases}
1 & \qquad \qquad \quad \qquad \qquad \quad \text{if}\ q \in \mathcal{U}_{p,\, m}, \\ 
0 &  \qquad \qquad \quad \qquad \qquad \quad \text{otherwise}. \\
\end{cases} \nonumber
\end{eqnarray} 
It is now easy to verify that these submatrices satisfy the relation \[T_{m} = H_{m - 1} + J_{m}^{T} X_{m}^{-1} J_{m}. \] 
\begin{remark}
For any $m \ge 0$, $(H_{m})_{pq} = 1$ if and only if $p \in V_{i}$ and $q \in \mathcal{U}_{p,\,i}$, for some $0 \le i \le m$.
\end{remark}

\medskip

\noindent
Every difference operator $H_{m}$ satisfies the properties enlisted in the following lemma. 

\noindent 
\begin{lemma} 
\label{prop of H_0} 
\begin{enumerate}
\item $H_{m}$ is a symmetric matrix with the row sum being zero for every row. 
\item The non-diagonal entries in $H_{m}$ are non negative; in particular either $1$ or $0$. \label{prop 2}
\item $H_{m}$ is non-positive definite with rank $N^{m + 1} - 1$. 
\item The function $u \in \ell(V_{m})$ is constant, if and only if $H_{m} u = 0$. \label{prop 4} 
\end{enumerate}
\end{lemma}
\begin{proof} The first two properties directly follow from the construction of the difference operator $H_{m}$. The remaining two can be easily proved by reducing the matrix to its row echelon form. 
\end{proof}

\section{Dirichlet forms on $V_{m}$} 
\label{Dirichlet forms}
\noindent 
A non-positive definite symmetric linear operator on any finite set $V$ satisfying the properties \eqref{prop 2} and \eqref{prop 4} mentioned in lemma \eqref{prop of H_0} gives rise to a Dirichlet form, a fundamental notion in the analysis on finite sets. Concerned readers may refer to \cite{kigaminetwork} for more details. A \emph{Dirichlet form} on $V$ is a non-negative definite symmetric bilinear form satisfying, 
\begin{enumerate}
\item  $\mathcal{E}\,(u, u) = 0$ if and only if $u$ is constant on $V$ and 
\item for any $u \in \ell(V),\ \mathcal{E}\,(u, u) \ge \mathcal{E}\, (\bar{u}, \bar{u})$ where $\bar{u}$ is defined by
\begin{equation}
\label{bar u}
\bar{u} (p)\ \ :=\ \ \begin{cases}
1 & \text{if  } \  u(p) \geq 1, \\
u(p) & \text{if  } \   0 < u(p) < 1, \\
0  &  \text{if  } \  u(p) \leq 0.
\end{cases}
\end{equation}
\end{enumerate} 
\noindent 
Now, returning to our setting of the shift space, the symmetric difference operator $H_{m}$ defined on $\ell(V_{m})$ in section \eqref{diff operator} naturally induces a symmetric bilinear form on $\ell(V_{m})$. We denote it by $\mathcal{E}_{H_{m}} $ and is given by,
\begin{equation}
\label{Dirichlet form on V_m}
\mathcal{E}_{H_{m}} (u,v)\ := \ -\left\langle u, H_{m} v  \right\rangle \ = \ -\sum\limits_{p \,\in \, V_{m}} u(p) \,H_{m} v(p).
\end{equation}
If $v = u$, for simplicity we denote $\mathcal{E}_{H_{m}} (u, u)$ by $\mathcal{E}_{H_{m}}(u)$. We verify that $ \mathcal{E}_{H_{m}} $ defined in such a way is a Dirichlet form on $ \ell(V_{m})$, once we observe the following:

\begin{proposition}
For $u, v \in \ell(V_{m})$, 
\begin{equation*} 
\mathcal{E}_{H_{m}}(u, v)\  = \ \frac{1}{2} \sum\limits_{p,q \,\in \, V_{m}} (H_{m})_{p q}  \left( u(p) -u(q) \right) \, \left( v(p) -v(q) \right).
\end{equation*}
\end{proposition}
\begin{proof}
Consider,
\begin{align*}
 \sum\limits_{p,q \, \in \,V_m}  (H_{m})_{pq}\, &\left( u(p)\,-\,u(q)  \right) \, \left( v(p)\,-\,v(q)  \right)\\
 &= \  \sum\limits_{p\,\in \,V_m} \Big( u(p) \, v(p)  \sum\limits_{q\,\in\, V_{m}}  (H_{m})_{pq} \Big) \ + \  \sum\limits_{q\,\in\, V_m} \Big( u(q)\, v(q)  \sum\limits_{p\,\in\, V_{m}}  (H_{m})_{pq} \Big) \\
 &\qquad - \sum\limits_{p \,\in\, V_m} \Big( u(p) \sum\limits_{q \,\in \,V_m}(H_{m})_{pq} \, v(q) \Big) \ - \ \sum\limits_{q \,\in\, V_m} \Big( u(q) \sum\limits_{p \,\in \,V_{m}}(H_{m})_{pq} \, v(p) \Big)
\end{align*}
Since the row sum and column sum of $H_{m}$ are zero, the first two terms in the expression on the right side above vanish. Also by the definition of $H_{m}$, we have $ H_{m} v (p)\, =  \sum\limits_{q \,\in \,V_{m}} (H_{m})_{p q} \, v(q)$. So by reversing the roles of $p$ and $q$ in the last term of above expression on the right side, we obtain,
\[ \frac{1}{2}\sum\limits_{p,q \, \in \,V_m}  (H_{m})_{pq}\, \left( u(p)\,-\,u(q)  \right) \, \left( v(p)\,-\,v(q)  \right) \ = \ - \sum\limits_{p \,\in\, V_m}  u(p)\, H_{m}v(p) \ = \ \mathcal{E}_{H_{m}} (u,v) \]
\end{proof}

\noindent
The following corollary follows when $v = u$, in the above proposition.
\begin{corollary}\label{formula E_Hm}
For any $u \in \ell(V_{m})$,
\begin{equation} 
\label{altrn rep of EHmu}
\mathcal{E}_{H_{m}}(u)\ \ =\ \ \frac{1}{2} \sum\limits_{p,q \,\in \, V_{m}} (H_{m})_{p q}  \left( u(p) -u(q) \right)^{2}.
\end{equation}
\end{corollary}

\noindent 
\begin{theorem}
The bilinear form $\mathcal{E}_{H_{m}} $ defined in equation \eqref{Dirichlet form on V_m} is a Dirichlet form on $V_{m}$.
\end{theorem}
\begin{proof}
Consider the alternate expression for $\mathcal{E}_{H_{m}}(u)$ obtained in the corollary \eqref{formula E_Hm}. Note that on the right hand side of equation (\ref{altrn rep of EHmu}), the terms corresponding to the points $p, q \in V_{m}$ such that $ p = q$ or $(H_{m})_{pq} = 0 $ contribute nothing to the sum. Among the remaining terms that contribute to the sum, the points $p, q$ are such that $ p \ne q$ with $(H_{m})_{pq} = 1 $.  Thus, all the terms in the sum are non-negative and we have,
\[ \mathcal{E}_{H_{m}}(u) \geq 0, \text{ for all } u \in \ell(V_{m}).  \] 
$ \mathcal{E}_{H_{m}}(u) = 0$ if and only if $(H_{m})_{pq}\, \left( u(p)-u(q)  \right)^2 = 0 $ for all $p,q \in V_{m} $, since every individual term in the sum for $\mathcal{E}_{H_{m}}(u) $ is non-negative. From the definition of $H_{m}$, it follows that whenever $(H_{m})_{pq} = 1 $, that there exists some $k \in \mathbb{N}$ such that $0 \le k \le m$ and $ q \in \mathcal{U}_{p, k}$. Then we obtain $(H_{m})_{pq}\, \left( u(p)-u(q)  \right)^2 = 0 $ if and only if $ u(q) = u(p)$ whenever $ q \in \mathcal{U}_{p, k}$ for some $0 \leq k \leq m$. In other words, $u$ assumes a constant value for any two $k$-related points in $V_{m}$. In particular, for any $p,\,q \in V_{0}, \ q \in \mathcal{U}_{p, 0} $ holds, and thus the function $u$ is constant on $V_{0}$. Recall that, any point in $V_{m}$ can be connected to a point in $V_{0}$ by a chain of related points at intermediary steps, as described in remark \eqref{chain connecting pts}. Therefore we obtain that $ \mathcal{E}_{H_{m}}(u) = 0$ if and only if $u$ is a constant function on $V_{m}$.
\medskip

\noindent
For a function $u \in \ell(V_{m})$, construct a function $\bar{u} \in \ell(V_{m})$ as defined in \eqref{bar u}. Consider,
\allowdisplaybreaks
\begin{eqnarray*}
\mathcal{E}_{H_{m}}\, (\bar{u}) & = & \frac{1}{2} \sum\limits_{p, q \,\in \,V_{m}} (H_{m})_{p q} \, \left( \bar{u}(p) - \bar{u}(q) \right)^{2} \\
& = & \frac{1}{2} \left[ \sum\limits_{\substack{p, q \,\in\, V_{m} \\ 0\, <\, u(p), u(q)\, < \,1}} (H_{m})_{p q} \, \left( u(p) - u(q) \right)^{2} \ +  \sum\limits_{\substack{p, q \,\in \, V_{m} \\ 1\, \le \, u(p) \\ 0 \,<\, u(q)\, < 1}} (H_{m})_{p q} \, \left( 1 - u(q) \right)^{2} \right. \\
& & \hspace{+2cm} \left. +\  \sum\limits_{\substack{p, q \,\in \, V_{m} \\ 1\, \le \, u(q) \\ 0\, <\, u(p)\, < \,1}} (H_{m})_{p q} \, \left( u(p) - 1 \right)^{2} \  +   \sum\limits_{\substack{p, q \,\in \, V_{m} \\ u(p)\, \le \, 0 \\ 0 \,<\, u(q)\, < 1}} (H_{m})_{p q} \, \left( u(q) \right)^{2} \right. \\
& & \hspace{+7cm} \left. +\ \sum\limits_{\substack{p, q \,\in \, V_{m} \\ u(q)\, \le \, 0 \\ 0 \,<\, u(p)\, < 1}} (H_{m})_{p q} \, \left( u(p) \right)^{2} \right] \\
& \le & \frac{1}{2} \sum\limits_{p, q \, \in \, V_{m}} (H_{m})_{p q} \, \left( u(p) - u(q) \right)^{2} \\
& = & \mathcal{E}_{H_{m}}\, (u) 
\end{eqnarray*}
\end{proof}
\medskip

\noindent
For any real valued function (not necessarily continuous) $u$ on $\Sigma_{N}^{+} $, denote its restriction to $V_{m}$ by $u|_{V_{m}}$. Clearly $u|_{V_{m}} \in \ell(V_{m}) $. The sequence $\left\lbrace \mathcal{E}_{H_{m}} (u|_{V_{m}}) \right\rbrace_{m \ge 0}$ is non-decreasing, as, for any $m \ge 0$, by the definition of $H_{m}$, we have,
\begin{eqnarray}
\mathcal{E}_{H_{m+1}}(u|_{V_{m+1}}) & = & \frac{1}{2} \ \sum\limits_{i\, =\, 0}^{m + 1} \ \sum\limits_{p \, \in \,V_{i}} \ \sum\limits_{q \, \in \, \mathcal{U}_{p,\,i}} \Big( u(p) - u(q) \Big)^{2} \notag \\
& = &\mathcal{E}_{H_{m}}(u|_{V_{m}})\, + \frac{1}{2}  \sum\limits_{p \, \in \,V_{m+1}} \  \sum\limits_{q \, \in \, \mathcal{U}_{p,\,m+1}} \Big( u(p) - u(q) \Big)^{2} \label{EH_m+1} \\
& \ge & \mathcal{E}_{H_{m}}(u|_{V_{m}}). \notag
\end{eqnarray}

\noindent
The following theorem states that the sequence $\{ (V_{m}, H_{m}) \}_{m \ge 0}$ that we have constructed, is compatible in the sense of equation \eqref{compatibility_energy}.

\begin{theorem} \label{extension}
Any $u_{m} \in \ell(V_{m}) $ can be uniquely extended to a function $u_{m+1} \in \ell(V_{m+1})$ preserving the respective Dirichlet forms in the sense that,
\[\mathcal{E}_{H_{m+1}} (u_{m+1})\  = \ \mathcal{E}_{H_{m}} (u_{m}) \ = \ \min \left\lbrace \mathcal{E}_{H_{m+1}} (v )\ |\  v \in \ell(V_{m+1}),\, v|_{V_{m}} = u_{m} \right\rbrace. \]
\end{theorem}
\begin{proof}
Assuming such an extension $u_{m+1}$ of $u_{m}$ exists, let us explicitly construct the same. Since $u_{m+1}$ should satisfy $\mathcal{E}_{H_{m+1}} (u_{m+1}) = \mathcal{E}_{H_{m}} (u_{m}) $, from equation \eqref{EH_m+1}, we get
\[ \sum\limits_{p \, \in \,V_{m+1}} \  \sum\limits_{q \, \in \, \mathcal{U}_{p,\,m+1}} \Big( u(p) - u(q) \Big)^{2} \ = \ 0, \]
which holds if and only if for any $p \in V_{m+1}$,
\[ u_{m+1}(q)\, = \,u_{m+1}(p) \ \text{ for all } \ q \in \mathcal{U}_{p,\,{m+1}}.  \] 
Thus, this extension is unique and it takes constant values on the equivalence classes $[p_{1}\,p_{2}\,\cdots p_{m + 1} ]|_{V_{m+1}}$. Recall from remark \eqref{m-related V_m-V_{m-1}}, that the deleted neighbourhood of any $p \in V_{m+1} \setminus V_{m}$ is given by, 
\[ \mathcal{U}_{p,\,{m+1}}\  = \ \left\lbrace q^{1},\, q^{2},\,\cdots, q^{N-1} \right\rbrace \ \subset \ [p_{1}\,p_{2}\,\cdots p_{m + 1} ]|_{V_{m + 1}},\]
with only one immediate neighbour $q^{N-1} \in V_{m}$. Therefore the extension $u_{m+1}$ on $V_{m+1} \setminus V_{m}$ is uniquely determined by $u_{m}$ as 
\[ u_{m+1}(p)\ = \ u_{m+1}(q^{1}) \ = \ \cdots \ = \ u_{m+1}(q^{N-2})\ = \ u_{m}(q^{N-1}). \] 
\end{proof}
\medskip

\section{Energy}
\label{energy}
The sequence of the Dirichlet forms corresponding to the compatible sequence of difference operators, as constructed in the last section, gives rise to a non-negative definite symmetric bilinear form in the limiting sense. We call this form as \textit{energy}. 
\begin{definition} 
The \textit{energy} of $u$ is defined as,
\[ \mathcal{E}(u) \ = \  \lim\limits_{m \to \infty} \mathcal{E}_{H_{m}}\,(u_{m}),  \] 
where $+\infty$ is a possible limiting value. Further, we define the \emph{domain of energy} as,
\[ dom \, \mathcal{E} \ := \ \left\lbrace  u \in \mathcal{C}(\Sigma_{N}^{+}) \ : \ \mathcal{E} (u) < \infty  \right\rbrace. \] 
\end{definition} 

\noindent 
Since $\mathcal{E}(u)$ is a series in which each summand is non-negative, $\mathcal{E}(u) = 0 $ if and only if $u$ is constant.\\ 

\noindent
We now see that the energy is a Markovian form. Consider a function $\phi : \mathbb{R} \longrightarrow \mathbb{R}$, 
\begin{equation*}
\phi(t) \ = \ \begin{cases}
0 \ \ \ \ \text{if} \ t \le 0 \\
t \ \ \ \ \text{if} \ t \in [0,1] \\
1 \ \ \ \ \text{if} \ t \ge 1. 
\end{cases}
\end{equation*}
Then, for any $u \in dom \, \mathcal{E}$, it follows directly from the definition of the energy $\mathcal{E}$, that $\phi \circ u \in dom \, \mathcal{E}$, with $\mathcal{E}(\phi\circ u) \le \mathcal{E}(u)$, proving the Markovian property. \\

\noindent
Similar to the way the energy is defined as the limit of finite Dirichlet forms, we define Laplacian as a renormalised limit of the finite difference operators, in the next section. At this juncture, we carefully restrict ourselves from calling this energy to be a Dirichlet form associated with the Laplacian (that we will soon define). Although the finite difference operators and finite Dirichlet forms are associated in a certain way, this relation may not get carried over in exactly the same way in the limit. But we can still expect some relation to exist between the energy and the Laplacian defined in the next section. However, this is not the aim of this paper and we will look into this matter in detail in a subsequent paper, \cite{sssnt}.
\medskip

\noindent
Consider the function $u_{m+1} \in \ell(V_{m+1})$ constructed in the proof of theorem \eqref{extension}. Extend the same to a function $u \in \mathcal{C}(\Sigma_{N}^{+})$ by fixing it to be constant on the cylinder sets of length $m+1$ in $\Sigma_{N}^{+}$. These constants are determined by the values of $u$ at the points in $V_{m}$. That is, 
\[ u(x)\ \ =\ \ u_{m} (p_{1}\, p_{2}\, \cdots\, p_{m}\, \dot{p}_{m + 1})\ \ \ \ \text{whenever}\ \ \ x \in \left[ p_{1}\, p_{2}\, \cdots\, p_{m}\, p_{m + 1} \right]. \] 
Then for all $n \ge m+1$, $\mathcal{E}_{H_{n}} (u|_{V_{n}}) = \mathcal{E}_{H_{m}} (u_{m})$. Due to the compatibility of the difference operators as proved in theorem \eqref{extension}, this particular extension has the least energy among all the extensions of $u_{m}$,
\[ \mathcal{E} (u) = \min \left\lbrace \, \mathcal{E}(v) \ : \ v \in dom \, \mathcal{E}, \, v|_{V_{m}} = u_{m} \, \right\rbrace .  \]
Therefore, we call such an extension of a function as the energy minimizer extension. In general we can define such functions as follows.
\begin{definition} 
A real valued function $h$ on $\Sigma_{N}^{+} $ is called as an \emph{energy minimizer}, if for some $n \ge 0$, 
\begin{equation*}
h(x) \ = \ h(p) \ \text{ whenever } \ x \in \left[p_{1}\,p_{2}\,\cdots \,p_{n}\,p_{n+1}  \right] ,
\end{equation*}
where $p = \left( p_{1}\,p_{2}\,\cdots \,p_{n}\,\dot{p}_{n+1} \right) \in V_{n}$.
\end{definition}

\noindent
The energy minimizer extension of a function in $\ell(V_{m})$ takes constant values on cylinder sets of length $m+1$. For instance, if $p = (p_{1}\, p_{2}\,\cdots\,p_{m}\,\dot{p}_{m+1}) \in V_{m}$ and $\chi_{p} \in \ell(V_{m})$ is its characteristic function, then its energy minimizer extension is given by $\chi_{p}^{m} : \Sigma_{N}^{+} \longrightarrow \mathbb{R}$ as,
\begin{equation} 
\label{uhe}
\chi_{p}^{m} \ = \  
\begin{cases}
1 &  \text{ on } \  [p_1\, p_2\, \cdots\, p_{m+1}] \\
0 &  \text{ elsewhere. }
\end{cases} 
\end{equation}

\noindent 
These are the simple functions in $\mathcal{C}(\Sigma_{N}^{+})$ which will play a crucial role in the study of Laplacian that we will define in the following section. Having defined the energy, let us look at the effective resistance on the set $V_{*}$. Our claim is that $V_{*}$ is unbounded with respect to the resistance metric. Recall the effective resistance between the points $a, b \in V_{*}$ is defined as 
\[R(a,b) =  \big[\, \min \left\lbrace \mathcal{E}(u)\, : \, u \in dom\, \mathcal{E},\, u(a)=1,\, u(b)=0    \right\rbrace \, \big] ^{-1} \]
Let $m \ge 1$ and choose two non related points $a = (a_{1}\, \cdots \, a_{m}\, \dot{a}_{m+1})$ and $b = (b_{1}\, \cdots \, b_{m}\, \dot{b}_{m+1})$ in $V_{m}\setminus V_{m-1}$, such that $a_{i} \neq a_{i+1}$, $b_{i} \neq b_{i+1}$ and $a_{i} \neq b_{i}$ for all $1 \le i \le m$. Consider the equivalence classes in the sets $V_{i}$, $1 \le i \le m$, generated from the points $a$ and $b$, 
\[ [a_{1}\, \cdots \, a_{m}]|_{V_{m}}, \ \  [a_{1}\, \cdots \, a_{m-1}]|_{V_{m-1}}, \ \ \cdots \  [a_{1}]|_{V_{1}}, \ \ [b_{1}\, \cdots \, b_{m}]|_{V_{m}}, \ \  [b_{1}\, \cdots \, b_{m-1}]|_{V_{m-1}}, \ \ \cdots \  [b_{1}]|_{V_{1}}. \]
By virtue of remark \eqref{chain connecting pts}, the points $a$ and $b$ can be connected by a chain of $i$-related points, for $0 \le i \le m$, each belonging to the equivalence classes above, and the points $(\dot{a}_{1})$ and $(\dot{b}_{1})$ in $V_{0}$. This is the longest chain of $i$-related points in $V_{m}$, connecting two non-related points in $V_{m}\setminus V_{m-1}$, due to the particular choice of the points $a$ and $b$.
\medskip

\noindent
We denote the points in each of these equivalence classes as, $[a_{1}\, \cdots \, a_{m}]|_{V_{m}} = \left\lbrace a, a_{m}^{1}, \cdots , a_{m}^{N-1} \right\rbrace $ with $a_{m}^{N-1} \in V_{m-1} \setminus V_{m-2}$. Similarly, $[a_{1}\, \cdots \, a_{m-1}]|_{V_{m-1}} = \left\lbrace a_{m}^{N-1}, a_{m-1}^{1}, \cdots , a_{m-1}^{N-1} \right\rbrace $ with $a_{m-1}^{N-1} \in V_{m-2} \setminus V_{m-3}$, and so on. Finally, $ [a_{1}]_{V_{1}} = \left\lbrace a_{2}^{N-1}, a_{1}^{1}, \cdots , a_{1}^{N-1} \right\rbrace $ with $a_{1}^{N-1} = (\dot{a}_{1}) \in V_{0}$. The points in the equivalence classes generated by the point $b$ are denoted in the same manner with $a$ replaced by $b$ in the above notation. 
\medskip

\noindent
Let us construct a function $u \in \ell(V_{m})$ as follows. Set $u(a) = 1$ and $u(b) = 0$. Let $\delta_{1}, \, \delta_{2} > 0 $ satisfying 
\[ \delta_{i}\ \ <\ \ \frac{1}{2m (N - 1)}\ \ \text{for}\ i = 1, 2\ \ \ \ \text{and}\ \ \ \ \delta_{1}^{2}\; +\; \delta_{2}^{2}\ \ <\ \ \frac{6 (m^{2} - m - 1)}{ (m^{4} + m) (N - 1)^{2} (2N^{2} - N)}. \]
At the points in the equivalence classes above generated by $a$, set the values of $u$ in the decreasing manner with a difference of $\delta_{1}$ as 
\begin{eqnarray*} 
u(a_{m}^{1}) = 1 - \delta_{1},\ u(a_{m}^{2}) = 1 - 2 \delta_{1}, & \cdots, & u(a_{m}^{N - 1}) = 1 - (N - 1) \delta_{1}, \\ 
& \cdots, & u(a_{m - 1}^{N - 1}) = 1 - 2(N - 1) \delta_{1}, \\ 
& \cdots, & u((\dot{a}_{1})) = 1 - m(N - 1) \delta_{1}. 
\end{eqnarray*}
Similarly, fix the values of $u$ at points in the equivalence classes generated by $b$, each increasing by the quantity $\delta_{2}$.
At all the remaining points of $V_{m}$, $u$ is set to take the constant value $u(\dot{a}_{1}) - \frac{\delta_{1}+\delta_{2}}{2}$. 
\medskip

\noindent
On careful observation, we note that only the points in the equivalence classes above and all the points in $V_{0}$ contribute in determining $\mathcal{E}_{H_{m}}(u)$ given by equation \eqref{altrn rep of EHmu}. The terms involving all other points vanish as the function $u$ is set to take the same constant value. Upon substituting for these values of $u$ in the expression of $\mathcal{E}_{H_{m}}(u)$ in equation \eqref{altrn rep of EHmu}, we obtain, \[ \mathcal{E}_{H_{m}}(u) < \frac{1}{m+1}.\]
Let $h(u)$ denote the energy minimizer taking constant values on cylinder sets of length $m+1$, obtained by extending $u$. Clearly, $(h(u))(a) = 1$ and $(h(u))(b) = 0$ and $\mathcal{E}(h(u)) = \mathcal{E}_{H_{m}}(u)$. Then,
\[ \big[\, \min \left\lbrace \mathcal{E}(u)\, : \, u \in dom \, \mathcal{E}, \, u(a)=1,\, u(b)=0    \right\rbrace \, \big]  \ \le \ \mathcal{E}(h(u)) \ < \ \frac{1}{m+1}. \] 
This implies that the effective resistance between $a$ and $b$ is, $R(a, b) > m+1$, thus proving our claim. As already discussed in the introduction, due to the unboundedness, the completion of $V_{*}$ with respect to $R$ will only be a proper subset of $\Sigma_{N}^{+}$. Thus the Laplacian on $\Sigma_{N}^{+}$ can not be obtained in the topology generated by the effective resistance. Despite this fact, in the next section we prove that the Laplacian of a continuous function in the standard topology induced by the metric $d$, can be defined, as a scalar limit of the difference operators $H_{m}$. 
\medskip

\section{The Laplacian} 
\label{the laplacian}

\noindent 
We begin this section by considering the discrete approximation of the Laplacian of a twice differentiable {\it i.e.,} $\mathcal{C}^{2}$ function on $\mathbb{R}$. If $u: \mathbb{R} \longrightarrow \mathbb{R}$ is a $\mathcal{C}^{2}$ function, then we write, 
\begin{equation} 
\label{laplacian in R}
\Delta u\, (x)\ \ =\ \ \lim_{h \to 0}\, \frac{1}{h^2}\, \left[ u\, (x + h) + u\, (x - h) - 2u\, (x) \right].
\end{equation}

\noindent 
Observe that the right hand side in the definition of $H_{m}$ as in equation \eqref{defn_of_Hm_u} is analogous to the quantity inside the bracket on right hand side of the equation \eqref{laplacian in R}. Hence, it makes sense to use this difference operator $H_{m}$ normalized appropriately and define the Laplacian on $\Sigma_{N}^{+}$. 
\medskip 

\noindent 
Now we use the density arguments to extend the operator on the full shift space, $\Sigma_{N}^{+}$. The Laplacian of a function $u$ in $\mathcal{C}(\Sigma_N^+)$ can now be defined as the limit of $H_{m}(u|_{V_{m}})$ with some proper scaling as given below. 
\begin{definition}
For the equidistributed Bernoulli measure $\mu$ defined in equation \eqref{measure}, define the set
\begin{eqnarray} 
\label{laplacian}
D_{\mu} & := & \Bigg\{ u \in \mathcal{C}(\Sigma_{N}^{+}) : \exists\, f \in \mathcal{C}(\Sigma_{N}^{+})\ \text{satisfying} \nonumber \\ 
& & \hspace{+2cm} \lim_{m \to \infty} \max_{p\, \in\, V_{m} \setminus V_{m - 1}} \left| \frac{H_{m} u (p)}{\mu([p_{1} p_{2} \cdots p_{m + 1}])}\, - f(p) \right| = 0 \Bigg\}. 
\end{eqnarray} 
Then, for $u \in D_{\mu} $, we write $f = \Delta_{\mu} u$. We call the operator $\Delta_{\mu}$ as the Laplacian on $\mathcal{C} (\Sigma_{N}^{+})$ and the set $D_{\mu}$ is referred to as the domain of the Laplacian. 
\end{definition}

\noindent
A function $h$ on $\Sigma_{N}^{+}$ is called a \emph{harmonic function}, if $\Delta h = 0$. A natural question that may arise in the readers' minds now, is whether the domain of the Laplacian $D_{\mu}$ is vacuous. Let $h \in \mathcal{C}(\Sigma_{N}^{+})$ be an energy minimizer, as defined in the previous section. There exists $n \ge 0$ such that $h$ is constant on the cylinder sets of length $n+1$. Consider the functions $ h_{m} \in \ell(V_{m}) $ given by $h_{m} = h|_{V_{m}}$. Then for any $m \ge n+1$, $ h_{m}(p)- h_{m}(q) = 0$ whenever $q \in \mathcal{U}_{p,\, m}$ and we have,
\[ N^{m+1} \sum\limits_{q \, \in \, \mathcal{U}_{p^{m},\, m} } \big( h_{m}(q)-h_{m}(p^{m}) \big) = 0 \ \ \text{for any} \ \ p^{m} \in V_{m}\setminus V_{m-1}. \]
Therefore, $\Delta h = 0$, which implies that every energy minimizer belongs to $D_{\mu}$ and is in fact a harmonic function. Also, $\mathcal{E}(h) = \mathcal{E}_{H_{n}}(h_{n})$ for some $n \ge 0$.

\medskip

\noindent
The convergence required in the definition of the Laplacian above, is relatively stronger to determine a function $f$ directly for a given $u \in D_{\mu}$. In the following theorem, we derive the pointwise formulation of the Laplacian which will come in handy in most of the calculations throughout. For convenience of notations, we lose the subscript  $\mu$ for the Laplacian, since we will only consider the equidistributed Bernoulli measure $\mu$.

\noindent 
\begin{theorem}
\label{pw laplacian}
Let $u \in D_{\mu}$ and $ \Delta u = f$. For any $ x \in \Sigma_{N}^{+}$, there exists a sequence of points $ \left\lbrace p^{m} \right\rbrace_{m \ge 1} $ with $p^{m}\, \in\, V_{m} \setminus V_{m - 1} $ such that
\begin{equation} 
\label{pointwise laplacian}
f(x)\ \ =\ \ \Delta u \,(x)\ \ =\ \ \lim_{m \rightarrow \infty} N^{m + 1}\, H_{m} u (p^{m}).
\end{equation}
\end{theorem}
\begin{proof}
Any point $x \in \Sigma_{N}^{+}$ looks like $x = (x_{1}\, x_{2}\,\cdots) $. For $m \ge 1$, consider the sequence of points $\left\lbrace p^{m} = (x_{1}\, x_{2}\, \cdots\, x_{m}\, \dot{l}) \right\rbrace_{m \ge 1}$, where $l \in S$ is chosen such that $l \ne x_{m}$. Such a selection of  $l$ guarantees that $p^{m}\, \in\, V_{m} \setminus V_{m - 1} $ for all $m \ge 1$. This sequence converges to the point $x$ in the metric $d$.  Thus, it directly follows from the definition of the Laplacian of $u$ as given in \eqref{laplacian} that, 
\[ \lim_{m \to \infty}  \left| N^{m + 1}\, H_{m} u (p^{m}) - f(p^{m}) \right|\ \ =\ \ 0. \] 
Then, by a simple use of triangle inequality we obtain the pointwise expression for the Laplacian as stated in equation \eqref{pointwise laplacian}. 
\end{proof}
\medskip

\noindent
The domain of the Laplacian and the domain of the energy both are dense linear subspaces of the space of continuous functions on $\Sigma_{N}^{+}$ as proved in the following theorem.
\begin{theorem} \label{dense inclusions}
\[ D_{\mu}  \ \subset \  dom \, \mathcal{E} \ \subset \  \mathcal{C} (\Sigma_{N}^{+}).\]
Moreover, both the inclusions are dense. 
\end{theorem}

\noindent
Before proving this theorem, we observe the following important fact.

\begin{lemma}\label{apxmtn by harm fn}
Any $u \in \mathcal{C}(\Sigma_{N}^{+}) $ is uniformly approximated by a sequence of harmonic functions. 
\end{lemma}
\begin{proof}
Let $u \in \mathcal{C}(\Sigma_{N}^{+})$. For each $m \ge 0$, define the functions $u_{m} \in \mathcal{C}(\Sigma_{N}^{+})$ as, 
\begin{equation}
\label{seq of harm fn}
u_{m} := \sum\limits_{p\, \in \, V_{m}} u(p)\, \chi_{p}^{m},
\end{equation}
where $\chi_{p}^{m}$ is as defined in equation \eqref{uhe}. Note that each $u_{m}$ is a harmonic function, being an energy minimizer. Since $u$ is uniformly continuous and $\Sigma_{N}^{+}$ is compact, the sequence $\{u_{m}\}_{m \ge 0}$ uniformly converges to $u$. Thus, every continuous function is approximated by a sequence of harmonic functions and the convergence is uniform.
\end{proof}

\medskip

\begin{proof} (Proof of Theorem \eqref{dense inclusions}) The inclusion $ dom \, \mathcal{E} \subset \mathcal{C} (\Sigma_{N}^{+})$ follows directly from the definition of $dom \, \mathcal{E}$. For the first inclusion, consider $u \in D_{\mu}$ and $f = \Delta u$. By theorem \eqref{pw laplacian}, for any $p \in V_{m} \setminus V_{m-1}$ we have, 
\[ \lim_{m \to \infty} N^{m+1}  \sum\limits_{q\, \in \,\mathcal{U}_{p,\,m}} \left[  u(q) - u(p) \right] \ < \ \infty.  \]
Since $u$ is continuous, there exist positive constants $C > 0$ and $M \in \mathbb{N}$, such that for all $m \ge M$,
\begin{equation}
\label{eq1}
| u(q) - u(p) | \ \le \ \frac{C}{N^{m+1}} \ \ \text{ whenever } \ \ d(q,p) \ \le \  \frac{1}{2^{m+1}}.
\end{equation}
Observe that the energy of the function $u$ can be written as,
\[\mathcal{E}(u)\ =  \ \mathcal{E}_{H_{M}}(u_{M}) \ + \ \lim_{n\, \to\, \infty} \sum\limits_{i \,= \,M+1}^{n} \ \sum\limits_ {p \,\in \,V_{i}} \ \sum\limits_{q \,\in \,\mathcal{U}_{p,\,i}}\left( u(q) - u(p)  \right)^{2}.  \]
Clearly for $q \in \mathcal{U}_{p,\,i} $ with $i \ge M+1$, we have, $ d(q,p) = \frac{1}{2^{i+1}} \le \frac{1}{2^{M+2}} $. Thus by equation \eqref{eq1}, we find a bound for $\mathcal{E}(u)$ as,
\begin{eqnarray*}
\mathcal{E}(u)\ & \le & \ \mathcal{E}_{H_{M}}(u_{M}) \ + \ \frac{C^{2} (N-1)}{2} \lim_{n\, \to\, \infty} \sum\limits_{i \,=\, M+1}^{n} \ \sum\limits_ {p \,\in \,V_{i}} \frac{1}{(N^{i+1})^2} \\
& = & \  \mathcal{E}_{H_{M}}(u_{M}) \ + \ \frac{C^{2} (N-1) }{2} \lim_{n\, \to\, \infty} \sum\limits_{i \,=\, M+1}^{n} \frac{1}{N^{i+1}} \\
& < & \ \mathcal{E}_{H_{M}}(u_{M}) \ + \ \frac{C^{2} (N-1) }{2} \sum\limits_{i \,=\, 1}^{\infty} \frac{1}{N^{i+1}} \\
& < & \ \infty.
\end{eqnarray*}
Therefore, $D_{\mu} \subset dom \, \mathcal{E}$. 
\smallskip 

\noindent
Lemma \eqref{apxmtn by harm fn} states that the harmonic functions are dense in the set $\mathcal{C}(\Sigma_{N}^{+}) $. Since the harmonic functions are the members of $D_{\mu}$, the inclusions in the statement of the theorem are dense.
\end{proof}
\medskip
 
\noindent
In the next part of this paper we proceed towards establishing the existence and uniqueness of a solution to the Dirichlet boundary value problem as stated in theorem \eqref{maintheorem}. We follow the standard approach to obtain the Green's function and the Green's operator which produces the required solution.

\section{Green's function} 
\label{Green's function}

\noindent 
In this section, we define the Green's function on $\Sigma_{N}^{+}$. Recall that the matrix $X_{m}$, as defined in equation (\ref{defn_of_X_m}) is symmetric and invertible. All the diagonal entries of $X_{m}$ are $-(N-1)$. Among the non-diagonal entries, the row corresponding to any point $p \in V_{m} \setminus V_{m-1}$ contains $1$ at $N-2$ places which correspond to the $N-2$ neighbours of $p$ in $U_{p,\,m}$. The remaining entries in the matrix $X_{m}$ are all $0$.

\noindent 
\begin{lemma} 
Consider the matrix $G_{m}\ : \  \ell (V_{m} \setminus V_{m - 1})  \longrightarrow  \ell (V_{m} \setminus V_{m - 1})$ defined by 
\begin{equation}
\label{G_m} 
(G_{m})_{pq} = 
\begin{cases}
\frac{2}{N} &\text{ if } q=p,\\
\frac{1}{N} &\text{ if } (X_{m})_{pq}=1,\\
0            &\text{ otherwise }. 
\end{cases}  
\end{equation}
Then, $X_{m}^{-1} = - G_{m}$. 
\end{lemma} 
\begin{proof} 
Let us first calculate the diagonal entries of $(X_{m})\times (-G_{m})$. For any $p \in V_{m} \setminus V_{m - 1} $ we get,
\allowdisplaybreaks
\begin{align*}
\left( (X_{m}) \,(-G_{m})  \right)_{pp}  &= \sum\limits_{ r\, \in \,V_{m} \setminus V_{m-1}} \left[  (X_{m})_{pr}\,   (-G_{m})_{rp}    \right] \\
&=  \ (X_{m})_{pp}\,   (-G_{m})_{pp} \,+\,   \sum\limits_{ r\, \in\, U_{p,\,m}} \left[  (X_{m})_{pr}\,   (-G_{m})_{rp}    \right] \\
&= \ -(N-1) \frac{-2}{N} \, +\, (N-2) \frac{-1}{N} \\
&= \ 1.
\end{align*}

\noindent
For the non-diagonal entries, consider any two distinct points $p,q \in V_{m} \setminus V_{m - 1} $. Then we have,
\begin{align}\label{X_mX_m^-1}
\left( (X_{m}) \,(-G_{m})  \right)_{pq}  &=  \ (X_{m})_{pp}\, (-G_{m})_{pq} \,+\,  (X_{m})_{pq}\, (-G_{m})_{qq} \,+\,\sum\limits_{ \substack{ r \, \in\, V_{m} \setminus V_{m-1} \\ r \,\ne \,q,p}}  \left[  (X_{m})_{pr}\, (-G_{m})_{rq}  \right] \nonumber\\
&= \ -(N-1)\, (-G_{m})_{pq} \,+\,  (X_{m})_{pq}\, \left( \frac{-2}{N} \right)  \,+\,\sum\limits_{ \substack{  r\, \in \,U_{p,\,m} \\ r\, \ne\, q}} (-G_{m})_{rq}.
\end{align}
If $q \in U_{p,\,m} $ then $(X_{m})_{pq}\, = \,1, \ (G_{m})_{pq} \,=\, (G_{m})_{rq} \,=\,  \frac{1}{N}$ whenever $r \in U_{p,\,m} $ with $r \ne q $. Substituting these values in equation \eqref{X_mX_m^-1}, we get
\[ \left( (X_{m}) \,(-G_{m})  \right)_{pq}\ =\ -(N - 1) \frac{-1}{N}\, +\, \frac{-2}{N}\, +\, (N - 3) \frac{-1}{N}\ =\ 0. \] 
If $q \notin U_{p,\,m} $ then $ (X_{m})_{pq}\, =\,(G_{m})_{pq} \,=\, 0  $. Now consider the third term in equation \eqref{X_mX_m^-1}. For $r \in U_{p,\,m}$, we know that $r \sim_{m} p$ with $r \ne q$. Since the $m$-relation $\sim_{m}$ is transitive, we have $r \notin U_{q,\,m}$ and $(G_{m})_{rq} \,=\,0 $. Therefore in this case too, we obtain
\[ \left[ (X_{m}) \,(-G_{m})  \right]_{pq} \ = \ 0. \]    
\end{proof} 

\noindent 
\begin{definition}
Let $G_{m} $ be the matrix as given in \eqref{G_m}. We define the \emph{Green's function} $ g : \Sigma_{N}^{+} \times \Sigma_{N}^{+} \longrightarrow \mathbb{R} \cup \{ \infty \}$ as,
\begin{eqnarray}
\label{greenfn}
g(x,y)\ \  = \ \ \begin{cases}
\sum\limits_{m\,=\,1}^{\rho(x,y)-1} \ \sum\limits_{r,s \,\in \,V_{m}  \setminus V_{m-1}} (G_{m})_{rs}\, \chi_{r}^{m}(x)\, \chi_{s}^{m}(y) &  \text{ if } \ \rho(x,y) > 1, \\
0 &  \text{ if } \ \rho(x,y) = 1.
\end{cases}
\end{eqnarray}
Here $ \rho(x,y)$ is the first instance where $x$ and $y$ disagree, as defined in section (\ref{sec_settings}). 
\end{definition}
\medskip 

\noindent 
\begin{lemma}\label{boundGF}
For any $ x, \, y \in \Sigma_{N}^{+}$ with $\rho(x,y) > 1$,
\[  0 \ \le \ g(x,y) \ \le \  \frac{2\,\rho(x,y)-3}{N}. \]
\end{lemma}

\begin{proof}
Let $ x = \left(x_{1}\,x_{2}\,\cdots \right), \, y = \left(y_{1}\,y_{2}\,\cdots \right) \in \Sigma_{N}^{+} $ such that $ \rho(x,y) \ge 2$. Since all the entries of the matrix $G_{m}$ are non-negative, it is clear that $ g(x,y) \ge 0$. We know that there exist unique points $r^{m} = (x_{1}\,x_{2}\, \cdots\,x_{m}\, \dot{x}_{m+1})$ and $s^{m} = (y_{1}\,y_{2}\, \cdots\,y_{m}\, \dot{y}_{m+1}) $ in $V_{m}$ such that $\chi_{r^{m}}^{m}(x)= 1$ and $\chi_{s^{m}}^{m}(y) = 1$. 
\medskip

\noindent
Suppose now that the point $x$ and $y$ are such that $x_{m} \ne x_{m+1}$ and $y_{m} \ne y_{m+1}$ for all $1 \le m \le \rho(x,y)-1$. Since $x_{m} = y_{m}$ for all $1 \le m \le \rho(x,y)-1$, we have,
\begin{eqnarray*}
r^{m} & = & s^{m} \  \ \text{ for all } \ 1 \le m \le \rho(x,y)-2 \\
r^{m} & \ne & s^{m} \ \ \ \ \text{ for }\ \ \ \ m = \rho(x,y)-1.
\end{eqnarray*}
Thus, 
\[ (G_{m})_{r^{m}s^{m}} \ = \ 
\begin{cases} 
\frac{2}{N} & \text{for all}\ 1 \le m \le \rho(x, y) - 2 \\ 
\frac{1}{N} & \text{when}\ m = \rho(x, y) - 1. 
\end{cases} \]
Substituting these values in equation \eqref{greenfn}, we get,
\[ g(x,y) \ \ = \ \  \frac{2}{N}\left( \rho(x,y)-2 \right) \,+\, \frac{1}{N} \ \ = \ \ \frac{2\,\rho(x,y)-3}{N}. \] 

\noindent
Suppose $x$ and $y$ are such that for some $1 \le m \le \rho(x,y)-1$, either $x_{m} = x_{m+1}$ or $y_{m} = y_{m+1}$. In this case either $r^{m} \notin V_{m} \setminus V_{m-1} $ or $s^{m} \notin V_{m} \setminus V_{m-1}$ respectively. In the definition of the Green's function as in equation \eqref{greenfn}, the terms corresponding to such values of $m$ do not contribute to the sum. Hence, we have the bound
\[ g(x,y)\ \leq \ \frac{2\,\rho(x,y)-3}{N}. \]
\end{proof}
\bigskip 

\noindent 
\begin{theorem}\label{prop of G.F.}
The Green's function satisfies the following properties:
\begin{enumerate}
\item For any $p \in V_{M}\setminus V_{M-1}$ and $y \in \Sigma_{N}^{+}, \ \ g(p,y) \ < \  \infty $.\vspace{-5mm}\\
\item For $x \in \Sigma_{N}^{+} \setminus V_{*},\ g(x,x)=\infty$.\vspace{-5mm}\\
\item The Green's function is continuous $(\mu \times \mu)$-almost everywhere. 
\end{enumerate}
\end{theorem}
\begin{proof}
\begin{enumerate}
\item Let $p \in V_{M}\setminus V_{M-1}$ and $m > M$. Then for any $r \in V_{m} \setminus V_{m-1},$ we have $\chi_{r}^{m}(p)=0$. So the sum in the equation (\ref{greenfn}) reduces to 
\begin{equation} \label{eq4} 
g(p,y) \ \ = \ \ \sum\limits_{m\,=\,1}^{M} \ \sum\limits_{r,s\,\in\, V_{m} \setminus V_{m-1}} (G_{m})_{rs}\, \chi_{r}^{m}(p)\, \chi_{s}^{m}(y) \ \ < \ \ \infty.
\end{equation}
\item If $x \in \Sigma_{N}^{+} \setminus V_{*}$, then $\rho(x,x)=\infty$ and thus $g(x,x)=\infty$.\\
\item Since $V_{*}$ is a countable set, $\mu(V_{*}) = 0$. We prove the continuity of the Green's function on the set $ \left(\Sigma_{N}^{+} \times \Sigma_{N}^{+}  \right) \setminus \left\lbrace (x,x)\, : \, x \in V_{*} \right\rbrace$. Let $ (x,y) \in \Sigma_{N}^{+} \times \Sigma_{N}^{+}$ and $(x^{n},y^{n})$ be a sequence converging to $(x,y)$. Then $\rho(x^{n},y^{n}) \rightarrow \rho(x,y)$ as $n \rightarrow \infty$.
\medskip 

\noindent
If $x = y \in \Sigma_{N}^{+} \setminus V_{*}$ then $g(x,y) = \infty $ and $g(x^{m}, y^{m}) \rightarrow \infty$ as $\rho(x^{m},y^{m}) \rightarrow \infty$.
\medskip

\noindent
If $x, y \in \Sigma_{N}^{+} $ such that $x \neq y$, then $\rho(x,y) < \infty$ and there exists some $M_{0} \in \mathbb{N}$ such that $\rho(x^{n},y^{n}) = \rho(x,y) $ for all $n \ge M_{0}$. Thus we obtain,
\begin{eqnarray*}
g(x^{n}, y^{n}) \ \ & = & \ \  \sum\limits_{m=1}^{\rho(x^{n},y^{n})-1} \sum\limits_{r,s \, \in \, V_{m}  \setminus V_{m-1}} (G_{m})_{rs}\, \chi_{r}^{m}(x^{n})\, \chi_{s}^{m}(y^{n}) \\
& = &  \sum\limits_{m=1}^{\rho(x,y)-1} \sum\limits_{r,s \, \in \, V_{m}  \setminus V_{m-1}} (G_{m})_{rs}\, \chi_{r}^{m}(x)\, \chi_{s}^{m}(y) \qquad \text{for all   }\  n \ge M_{0}. \\
\end{eqnarray*}
Thus, $ g(x^{n}, y^{n}) \to g(x,y)$ as $n \rightarrow \infty$ in both the cases, proving the almost everywhere continuity of $g$. 
\end{enumerate}
\end{proof}

\section{Green's operator}
\label{Green's operator}

\noindent 
In this section we define the Green's operator and study some of its properties.


\noindent 
\begin{definition} 
Let $L^{1}(\Sigma_{N}^{+})$ be the space of $\mu$-integrable functions on $\Sigma_{N}^{+}$. We define the Green's operator on $L^{1}(\Sigma_{N}^{+}) $ as an integral operator whose kernel is the Green's function as,
\begin{equation*}
G_{\mu}f(x)\ \ := \ \ \int_{\Sigma_N^+ \setminus \{x \} } g(x,y)\, f(y) \,\mathrm{d} \mu(y) \quad \text{ for } \ f \in L^{1}(\Sigma_{N}^{+}) . 
\end{equation*} 
\end{definition} 

\noindent
As we proved in the last section, for any $x \in \Sigma_{N}^{+} \setminus V_{*},\ g(x,x) = \infty$. Since the points have no mass, we remove the point $x$ from the domain of the integration in the definition above.
\bigskip 

\noindent 
  \begin{theorem}\label{prop of G.O.}
Let $ f \in L^{1}(\Sigma_{N}^{+})$. The Green's operator satisfies the following:
\begin{enumerate}
\item $G_{\mu}f \in L^{1}(\Sigma_{N}^{+})$. 
\item If $f \in \mathcal{C}(\Sigma_{N}^{+})$, then $G_{\mu}f \in \mathcal{C}(\Sigma_{N}^{+})$. 
\item $G_{\mu}f|_{V_{0}}\ =\ 0$.
\end{enumerate}
\end{theorem}
\begin{proof}
\begin{enumerate}
\item Fix $M \in \mathbb{N}$. Observe that,
\begin{equation*}
|G_{\mu}f(x)|\  \le \int\limits_{\Sigma_{N}^{+} \setminus \left[x_{1}\,\cdots\,x_{M} \right]} |\,g(x,y)\,f(y)\,|\,\mathrm{d}\mu(y)\ +\ \int\limits_{\left[x_{1}\,\cdots\,x_{M} \right]\setminus \{x\}} |\,g(x,y)\,f(y)\,|\,\mathrm{d}\mu(y).
\end{equation*} 
For any $y \in \Sigma_{N}^{+} \setminus \left[x_{1}\,\cdots\,x_{M} \right]$, we have $\rho(x,y) \le M$ and thus by lemma \eqref{boundGF} we have $|g(x,y)| \le \frac{2M - 3}{N}$. Since $ f \in L^{1}(\Sigma_{N}^{+})$, the first integral above can be bounded by,
\begin{align*}
\int\limits_{\Sigma_{N}^{+} \setminus \left[x_{1}\,\cdots\,x_{M} \right]} |\,g(x,y)\,f(y)\,|\,\mathrm{d}\mu(y) \ \ & \le\ \  \frac{2M-3}{N} \ \int\limits_{\Sigma_{N}^{+} \setminus \left[x_{1}\,\cdots\,x_{M} \right]} |f(y)| \,\mathrm{d}\mu(y)\\
& \le \ \ \frac{2M-3}{N} \|f\|_{L^{1}} \\ 
& < \ \  \infty.
\end{align*}
Let us now look at the second integral. For $i \ge 1$, consider the sets 
\[ C_{i}:=\bigcup\limits_{\substack{y_{M+i}\, \in \, S \\ y_{M+i} \, \ne \, x_{M+i} }} \  \left[x_{1}\,\cdots\,x_{M+i-1}\,y_{M+i}  \right].  \]
Note that $\left\lbrace  C_{i} \, : \, i \ge 1 \right\rbrace $ forms a partition of $\left[x_{1}\,\cdots\,x_{M} \right] \setminus \{x\}$. As $y_{M+i}$ can be any of the $N-1$ symbols from $S$ other than $x_{M+i}$, $\mu(C_{i})= \frac{N-1}{N^{M+i}}$. Since $\rho(x,y) = M+i$ for any $y \in C_{i}$, again by lemma \eqref{boundGF} we have $|g(x,y)| \le \frac{2(M+i)-3}{N}$. Thus we obtain a bound for the second integral as,
\begin{align*}
\int\limits_{\left[x_{1}\,\cdots\,x_{M} \right]\setminus \{x\}} |\,g(x,y)\,f(y)\,|\,\mathrm{d}\mu(y)\ \ & \le \ \ \| f\|_{L^{1}} \left[ \sum\limits_{i \ge 1} \int_{C_{i}} \frac{\left( 2(M+i)-3 \right) }{N} \,\mathrm{d}\mu(y) \right]\\
& =\ \ \frac{\| f\|_{L^{1}}\,(N-1)}{N} \left[ \sum\limits_{i \ge 1} \frac{2(M+i)-3}{N^{M+i}} \right] \\ 
& < \ \ \infty.
\end{align*}
The ratio test ensures that the series inside the bracket converges. The bounds for both the integrals are uniform and therefore, $G_{\mu}f \in L^{1}(\Sigma_{N}^{+})$.

\item Let $f \in \mathcal{C}(\Sigma_{N}^{+})$. Then $f$ is bounded and $\| f \|_{\infty} = \sup\limits_{x \in \Sigma_{N}^{+}} |f(x)| \, < \, \infty $. We are interested to prove that $G_{\mu}f$ is continuous at any point $x = (x_{1} \, x_{2}\, \cdots) \in \Sigma_{N}^{+}$. Let us take a sequence of points $\{ x^{n} \}_{n \in \mathbb{N}}$ such that $x^{n} \rightarrow x$ in $\Sigma_{N}^{+}$, and $x^{n} \ne x$ for any $n \ge 1$. Therefore for every $n \ge 1$, there exists $M\in \mathbb{N}$ (which depends on $n$) such that $d(x^{n}, x) = \frac{1}{2^{M+1}}$ or equivalently, $x^{n} \in [x_{1}\, \cdots\, x_{M}]$ and $\rho(x^{n},x) = M+1$. Clearly, as $n \to \infty$, $x^{n} \to x $ and $M \to \infty$. Consider,
\begin{equation}\label{continuity G_muf1}
\big| G_{\mu}f(x) -  G_{\mu}f(x^{n}) \big|\  \le  \int\limits_{\Sigma_{N}^{+} \setminus \{x,\,x^{n} \} } \big| g(x,y) - g(x^{n},y)\big| \, \big|f(y)\big| \,\mathrm{d} \mu(y). 
\end{equation}
We now analyse the term $\big|g(x,y) - g(x^{n},y)\big|$ for all possible combinations of $\rho(x,y)$ and $\rho(x^{n},y)$. First, if $\rho(x,y) = 1$ then $\rho(x^{n},y) = 1$ and thus by the definition of the Green's function, $g(x,y) = g(x^{n}, y) = 0$.  If $ 1 < \rho(x,y) \le M$ then $\rho(x^{n},y) = \rho(x,y) = \rho$ (say). Therefore,
\begin{equation*}
g(x,y)- g(x^{n},y) \ = \ \sum\limits_{m=1}^{\rho-1} \sum\limits_{r,s \in V_{m}  \setminus V_{m-1}} \left[ (G_{m})_{rs}\, \chi_{r}^{m}(x)\, \chi_{s}^{m}(y) \ - (G_{m})_{rs}\, \chi_{r}^{m}(x^{n})\, \chi_{s}^{m}(y) \right].
\end{equation*}
In this case, since $\rho \le M$ and $x$ and $x^{n}$ agree on the initial $M$ places, we have for every $m \le \rho-1$, $\chi_{r}^{m}(x) = 1$ if and only if $\chi_{r}^{m}(x^{n}) =1$. Therefore all the terms in this sum get cancelled and we get $g(x,y)- g(x^{n}, y)= 0$. In short, for any $y \notin [x_{1}\, \cdots\, x_{M}]$, $\big|g(x,y) - g(x^{n},y)\big| = 0$. Thus the integration in equation \eqref{continuity G_muf1} reduces to the following integration on the set $[x_{1}\, \cdots\, x_{M}]\setminus \{x,\,x^{n} \}$, 
\begin{equation}\label{continuity G_muf2}
\big| G_{\mu}f(x) -  G_{\mu}f(x^{n}) \big|\  \le  \int\limits_{[x_{1}\, \cdots\, x_{M}] \setminus \{x,\,x^{n} \} } \big| g(x,y) - g(x^{n},y)\big| \, \big|f(y)\big| \,\mathrm{d} \mu(y). 
\end{equation}
Set $A:= [x_{1}\, \cdots\, x_{M}]\setminus \{x^{n},\,x \}$. For any $y \in A$, $y \ne x^{n}$ and $y \ne x$. This implies, $\rho(x,y)$ and $\rho(x^{n},y)$ are finite and $|\rho(x,y) - \rho(x^{n},y)| = k < \infty$, for some $k \in \mathbb{N}$. Let us define the sets 
\[A_{k} := \left\lbrace y \in A \, : \,  |\rho(x,y) - \rho(x^{n},y)| = k  \right\rbrace. \]
Clearly, all $A_{k}'s$ are mutually disjoint and $A = \bigcup\limits_{k \ge 0} A_{k} $.\\

\noindent
Let us first consider $y\in A_{0}$. 
\medskip 

\noindent 
\textbf{Case \RomanNumeralCaps{1}:} Suppose $\rho(x,y)=\rho(x^{n},y)= M+1$. Then $y \in [x_{1}\, \cdots\, x_{M}]$ such that $y_{M+1} \ne x_{M+1}$ and $y_{M+1} \ne x^{n}_{M+1}$. 
\medskip 

\noindent 
\textbf{Case \RomanNumeralCaps{2}:} Suppose $y \in A$ such that $\rho(x,y) \ge M+2$. Then $y_{M+1} = x_{M+1}$ which implies $y_{M+1} \ne x^{n}_{M+1}$ and $\rho(x^{n},y)= M+1$. For such a choice of $y$, $ |\rho(x,y) - \rho(x^{n},y)| = 1$, thus $y \notin A_{0}$. Note that the roles of $x$ and $x^{n}$ are interchangeable, so by the same argument, any $y \in A$ satisfying $\rho(x^{n},y) \ge M+2$ can not belong to $A_{0}$. 
\medskip 

\noindent 
This tells us that any $y \in A_{0}$ satisfies the condition in case I. The measure of $A_{0}$ is then obtained as,
\[ \mu (A_{0}) \ = \ \frac{1}{N^{M}} \frac{N-2}{N}.\] 
Now, for $y \in A_{0}$, consider,
\begin{eqnarray*}
\big| g(x,y) - g(x^{n},y) \big| & = & \sum\limits_{m=1}^{M} \ \sum\limits_{r,s\, \in \, V_{m}  \setminus V_{m-1}} (G_{m})_{rs} \,\chi_{r}^{m}(x)\, \chi_{s}^{m}(y) \\
& \ & \ \ \ \  -\sum\limits_{m=1}^{M} \ \sum\limits_{r,s \,\in \, V_{m}  \setminus V_{m-1}} (G_{m})_{rs}\, \chi_{r}^{m}(x^{n})\, \chi_{s}^{m}(y).
\end{eqnarray*}
Here, $x, x^{n}$ and $y$ agree on first $M$ coordinates. Thus, all the terms in the above sum get cancelled for $1 \le m \le M - 1$, and we are left with the terms only corresponding to $m = M$.
\begin{eqnarray*}
\big| g(x,y) - g(x^{n},y) \big| & = &  \sum\limits_{r,s \, \in \, V_{M}  \setminus V_{M-1}} (G_{M})_{rs} \,\chi_{r}^{M}(x)\, \chi_{s}^{M}(y)\\
& \ & \ \ \ \  -\sum\limits_{r,s\, \in \, V_{M}} (G_{M})_{rs}\, \chi_{r}^{M}(x^{n})\, \chi_{s}^{M}(y).
\end{eqnarray*}
In the term corresponding to $x$, $\chi_{r}^{M}(x) = 1$ and $\chi_{s}^{M}(y)=1$ if and only if $r = (x_{1}\,\cdots\, x_{M}\, \dot{x}_{M+1})$ and $s = (x_{1}\,\cdots\, x_{M}\, \dot{y}_{M+1})$. Here $r \ne s$, and depending on whether $r$ and $s$ belong to $V_{M}  \setminus V_{M-1}$ or not, the only possible values of corresponding $(G_{M})_{rs}$ are $\frac{1}{N}$ or $0$. Further, the minimum value that the term $\sum\limits_{r,s\, \in \, V_{M}} (G_{M})_{rs}\, \chi_{r}^{M}(x^{n})\, \chi_{s}^{M}(y)$ can take, is $0$. Therefore,
\[\big| g(x,y) - g(x^{n},y) \big| \le \frac{1}{N} \ \ \text{for all} \ \ y \in A_{0}. \]
\smallskip

\noindent
Let us now fix $k \ge 1$ and consider $y \in A_{k}$. 
\medskip 

\noindent 
\textbf{Case \RomanNumeralCaps{3}:} Suppose $\rho(x^{n},y)= M+1+k$. Then $y \in [x_{1}\, \cdots\, x_{M}]$ such that $y_{M+1} \ne x_{M+1}$ and thus $\rho(x,y) = M+1$. 
\medskip 

\noindent 
\textbf{Case \RomanNumeralCaps{4}:} Similar to the case III, another possible choice for $y$ is when $\rho(x,y)= M+1+k$ and $\rho(x^{n},y)= M+1$. 
\medskip 

\noindent 
\textbf{Case \RomanNumeralCaps{5}:} Suppose $y \in A$ such that $ \rho(x,y) \ne M+1$ and $\rho(x,y) \ne M+1+k$. Then $\rho(x^{n}, y) = M+1$ which results in $|\rho(x,y)- \rho(x^{n}, y)| \ne k$ and $y \notin A_{k}.$ Same argument holds if $\rho(x^{n},y) \ne M+1$ and $\rho(x^{n},y) \ne M+1+k$.  
\medskip 

\noindent 
This establishes that any $y \in A_{k}$ will satisfy the condition of either case \RomanNumeralCaps{3} or case \RomanNumeralCaps{4}. Therefore the measure of $A_{k}$ can be given by,
\[ \mu (A_{k}) \ = \ \frac{2}{N^{M+k}} \frac{N-1}{N}.\]

Let $y \in A_{k}$ which belongs to case \RomanNumeralCaps{4}, that is, $\rho(x,y)= M+1+k$ and $\rho(x^{n},y)= M+1$. We get the same result for case \RomanNumeralCaps{3} as well, so it is enough to work with case \RomanNumeralCaps{4}. Consider,
\begin{eqnarray*}
\big| g(x,y) - g(x^{n},y) \big| & = & \sum\limits_{m=1}^{M+k} \  \sum\limits_{r,s\, \in \, V_{m}  \setminus V_{m-1}} (G_{m})_{rs} \,\chi_{r}^{m}(x)\, \chi_{s}^{m}(y) \\
& \ & \ \ \ \  -\sum\limits_{m=1}^{M} \sum\limits_{r,s \,\in \, V_{m}  \setminus V_{m-1}} (G_{m})_{rs}\, \chi_{r}^{m}(x^{n})\, \chi_{s}^{m}(y).
\end{eqnarray*}
As discussed before, all the terms corresponding to $1 \le m \le M-1 $ in the above expression get cancelled and the above sum reduces to,
\begin{eqnarray*}
\big| g(x,y) - g(x^{n},y) \big| & = & \sum\limits_{m=M}^{M+k} \ \sum\limits_{r,s\, \in \, V_{m}  \setminus V_{m-1}} (G_{m})_{rs} \,\chi_{r}^{m}(x)\, \chi_{s}^{m}(y) \\
& \ & \ \ \ \  -\sum\limits_{r,s\, \in \, V_{M}} (G_{M})_{rs}\, \chi_{r}^{M}(x^{n})\, \chi_{s}^{M}(y).
\end{eqnarray*} 
Consider the term corresponding to $x$ in the above sum. Observe that for $M \le m \le M+k-1$, $\chi_{r}^{m}(x) = 1$ and $\chi_{s}^{m}(y)=1$ if and only if $r = s = (x_{1}\,\cdots\, x_{m}\, \dot{x}_{m+1}) \in V_{m} \setminus V_{m-1}$. Therefore $(G_{m})_{rs} = \frac{2}{N}$ or $0$. For $m = M+k$, $\chi_{r}^{M+k}(x) = 1$ and $\chi_{s}^{M+k}(y)=1$ if and only if $r = (x_{1}\,\cdots\, x_{M+k}\, \dot{x}_{M+k+1})$ and $s = (x_{1}\,\cdots\, x_{M+k}\, \dot{y}_{M+k+1})$. Clearly $r \ne s$ and depending on whether $r$ and $s$ belong to $ V_{M+k} \setminus V_{M+k-1}$ or not, we get $(G_{m})_{rs} = \frac{1}{N}$ or $0$. The minimum value of the term in the above sum corresponding to $x^{n}$ is $0$. Substituting for these values we get,
\[\big| g(x,y) - g(x^{n},y) \big| \le \frac{2k}{N}+\frac{1}{N} < \frac{2(k+1)}{N} \ \ \text{for all} \ \ y \in A_{k}. \]
\medskip

\noindent
It is easy to verify that $\mu(A) = \sum\limits_{k \, \ge \, 0} \mu (A_{k})$. Let us now evaluate the required integration from the equation \eqref{continuity G_muf2}.
\begin{eqnarray*}
\big| G_{\mu}f(x) -  G_{\mu}f(x^{n}) \big|\  & \le & \int\limits_{A} \big| g(x,y) - g(x^{n},y)\big| \, \big|f(y)\big| \,\mathrm{d} \mu(y) \\
& = & \sum\limits_{k \ge 0}\ \int\limits_{A_{k} \setminus \{x,x^{n} \} } \big| g(x,y) - g(x^{n},y)\big| \, \big|f(y)\big| \,\mathrm{d} \mu(y) \\
& < & \frac{1}{N}  \|f\|_{\infty} \left( \frac{N-2}{N^{M+1}} \right)  + \ \sum\limits_{k \ge 1} \frac{2(k+1)}{N}  \|f\|_{\infty}  \left( \frac{2(N-1)}{N^{M+1+k}} \right) \\
& = & \frac{1}{N^{M+2}}  \|f\|_{\infty} \left( N-2 \, + \, 4(N-1)\sum\limits_{k \ge 1} \frac{k+1}{N^{k}} \right) \\
& = & \frac{C}{N^{M+2}},
\end{eqnarray*} 
where the ratio test guarantees the convergence of the series $\sum\limits_{k \ge 1} \frac{k+1}{N^{k}}$, and the constant $C$ depends only on $f$ and $N$. Therefore, as $n \to \infty$, $M \to \infty$ and we can conclude that $\big| G_{\mu}f(x) -  G_{\mu}f(x^{n}) \big| \to 0$, proving the continuity of $G_{\mu}f$.

\item Let $(\dot{l}) \in V_{0}$. For any $m \ge 1$ and $r \in V_{m} \setminus V_{m-1}, \  \chi_{r}^{m}(\dot{l}) = 0$. Therefore, $g((\dot{l}),y) = 0$ for any $y \in \Sigma_{N}^{+}$. Thus we have $G_{\mu}f|_{V_{0}} = 0$.
\end{enumerate}
\end{proof}
\medskip

\section{The solution to the BVP in theorem \eqref{maintheorem}}
\label{The Dirichlet problem} 

The objective of this section is to find the solutions to the analogous Dirichlet  boundary value problem on the full one-sided shift space $\Sigma_{N}^{+}$. We begin with the following two lemmas. 

\begin{lemma}\label{H_mG_mu}
For any $n \ge 1$ and $p \in V_{n}\setminus V_{n-1}$,
\begin{equation*} 
H_{n} G_{\mu} f(p)= - \int\limits_{\Sigma_{N}^{+}} \chi_{p}^{n}\, f \,\mathrm{d} \mu. 
\end{equation*}
\end{lemma}
\begin{proof}
Let $q^{1},\,q^{2},\,\cdots,\, q^{N-1} \in \mathcal{U}_{p,\,n}$ be as defined in section \eqref{sec_settings}. Using the definition of $H_{m}$ as in equation (\ref{defn_of_Hm_u}) and the definition of Green's operator we obtain, 
\begin{align*}
H_{n} G_{\mu} f(p)\ \ &= \ \ -(N-1)\, G_{\mu} f(p)\,+\, G_{\mu} f(q^{1}) \,+\, G_{\mu} f(q^{2}) \,+\, \cdots \,+\, G_{\mu} f(q^{N-1}) \\
&= \ \ \int\limits_{\Sigma_{N}^{+} \setminus \{ p,\,q^{1}, \,\cdots,\, q^{N-1}  \}} \left[-(N-1)\,g(p,y)\,+\, g(q^{1},y)\,+\,\cdots \right. \\
& \hspace{+7cm} \left. +\, g(q^{N-1},y)\right]\,f(y)\,\mathrm{d}\mu(y).
\end{align*}
We only need to prove \[ -(N-1)\,g(p,y)\,+\, g(q^{1},y)\,+\,\cdots \,+\, g(q^{N-1},y)\  = \ -\chi_{p}^{n}(y). \] We know that $p,\,q^{1},\,\cdots,\, q^{N-2} \in V_{n}\setminus V_{n-1}$ and $q^{N-1}\in V_{n-1}$. Therefore by equation \eqref{eq4} in the proof of Theorem \eqref{prop of G.F.} we have,
\begin{align} \label{eq5}
 -\,(N-1)\,g(p,y)\,&+\, g(q^{1},y)\,+\,\cdots \,+\, g(q^{N-1},y) \nonumber  \\
= -\,&(N-1) \left[ \sum\limits_{m\,=\,1}^{n} \ \sum\limits_{r,s\,\in \,V_{m} \setminus V_{m-1}} (G_{m})_{rs}\, \chi_{r}^{m}(p)\, \chi_{s}^{m}(y)  \right]\ \nonumber  \\ 
&+\sum\limits_{m\,=\,1}^{n}\ \sum\limits_{r,s\,\in\, V_{m} \setminus V_{m-1}} (G_{m})_{rs}\, \chi_{r}^{m}(q^{1})\, \chi_{s}^{m}(y) \nonumber  \\
  &+ \cdots \ +\   \sum\limits_{m\,=\,1}^{n} \ \sum\limits_{r,s\,\in\, V_{m} \setminus V_{m-1}} (G_{m})_{rs}\, \chi_{r}^{m}(q^{N-2})\, \chi_{s}^{m}(y) \nonumber  \\
&+\ \sum\limits_{m\,=\,1}^{n-1} \ \sum\limits_{r,s\,\in\, V_{m} \setminus V_{m-1}} (G_{m})_{rs}\, \chi_{r}^{m}(q^{N-1})\, \chi_{s}^{m}(y) \nonumber \\
= -\,&(N-1) \sum\limits_{ r,s\,\in\, V_{n} \setminus V_{n-1} } (G_{n})_{rs}\, \chi_{r}^{n}(p)\, \chi_{s}^{n}(y)   \nonumber  \\
&+ \sum\limits_{r,s \,\in\, V_{n} \setminus V_{n-1}} (G_{n})_{rs}\, \chi_{r}^{n}(q^{1})\, \chi_{s}^{n}(y)\ \\
&+ \cdots \,+ \sum\limits_{r,s \,\in \,V_{n} \setminus V_{n-1}} (G_{n})_{rs}\, \chi_{r}^{n}(q^{N-2})\, \chi_{s}^{n}(y)  \nonumber .
\end{align}
Here note that, for $1 \le m \le n-1$, all the terms get cancelled and only the terms with $m = n$ remain. We now examine when this term will be nonzero. In the first term in equation \eqref{eq5}, $ \chi_{r}^{n}(p) = 1$ iff $r = p \in V_{n} \setminus V_{n-1}$. In that case, $s = p,\, q^{1},\cdots q^{N-2}$ so that $(G_{n})_{rs} \ne 0$. Similar in the second term $ \chi_{r}^{n}(q^{1}) = 1$ iff $r = q^{1} \in V_{n} \setminus V_{n-1}$ and again $s = p,\, q^{1},\cdots q^{N-2}$ so that $(G_{n})_{rs} \ne 0$. And so on for the remaining terms. 
\medskip 

\noindent 
After substituting the values of $(G_{n})_{rs}$ for the corresponding choices of $r$ and $s$ in each term in equation \eqref{eq5} above, we get,
\allowdisplaybreaks
\begin{align*}
-\,(N - 1) \,&g(p,y)\,+\, g(q^{1},y)\,+\,\cdots \,+ \,g(q^{N-1},y) \\
=\  -\, &(N-1) \left[ \frac{2}{N} \chi_{p}^{n}(y)\, +\, \frac{1}{N} \chi_{q^{1}}^{n}(y) \,+\, \cdots \,+\, \frac{1}{N} \chi_{q^{N-2}}^{n}(y) \right] \\
&\ + \left[ \frac{1}{N} \chi_{p}^{n}(y)\, +\, \frac{2}{N} \chi_{q^{1}}^{n}(y) \,+\, \cdots \,+\, \frac{1}{N} \chi_{q^{N-2}}^{n}(y) \right] \\
 &\ +\  \cdots\ + \\ 
 &\ +\ \left[ \frac{1}{N} \chi_{p}^{n}(y)\, + \,\frac{1}{N} \chi_{q^{1}}^{n}(y)\, +\, \cdots\, +\, \frac{2}{N} \chi_{q^{N-2}}^{n}(y) \right]\\
= \quad &\chi_{p}^{n}(y) \left[ \frac{-2(N-1)}{N}\,+\,\underbrace{\frac{1}{N}\, + \,\cdots \,+\, \frac{1}{N} }_{N-2 \text{ terms }}  \right]\\
 &\ +\ \chi_{q^{1}}^{n}(y) \left[ \frac{-(N-1)}{N}\, +\, \frac{2}{N}\, +\,\underbrace{\frac{1}{N} \,+\, \cdots\,+\,\frac{1}{N} }_{N-3 \text{ terms }}  \right]\\
 &\ +\  \cdots\ + \\
 &\ +\ \chi_{q^{N-2}}^{n}(y) \left[ \frac{-(N-1)}{N}\, +\, \underbrace{\frac{1}{N} \,+ \,\cdots\,+\,\frac{1}{N} }_{N-3 \text{ terms }} \,+\, \frac{2}{N}  \right]\\
=\ \ \ &-\chi_{p}^{n}(y).
\end{align*}
\end{proof}
\bigskip 

\noindent 
\begin{lemma}\label{deltaG_muf=-f}
For any $f \in \mathcal{C}(\Sigma_{N}^{+}) $, we have $G_{\mu}f \in D_{\mu}$ and $\Delta\, (G_{\mu}f) = -f $.
\end{lemma}
\begin{proof}
Let $m \ge 1$ and $p \in V_{m} \setminus V_{m-1}$. With the help of Lemma \eqref{H_mG_mu}, we get,
\allowdisplaybreaks
\begin{align*}
 &\left| N^{m+1}\, H_{m} G_{\mu}f(p) \,+\, f(p)  \right| \\
&= \left| -\int\limits_{\Sigma_{N}^{+}}N^{m+1} \, \chi_{p}^{m}(y) \,f(y)\,\mathrm{d} \mu(y) \ +\  \int\limits_{\Sigma_{N}^{+}}N^{m+1} \, \chi_{p}^{m}(y) \,f(p) \,\mathrm{d} \mu(y)   \right| \\
&\le \int\limits_{[p_{1}\,p_{2}\,\cdots\, p_{m+1}]}N^{m+1}\, \left| f(p)-f(y) \right| \,\mathrm{d} \mu(y)\\
&< \epsilon_{m},
\end{align*}
where $\epsilon_{m} := \sup\limits_{y \in [p_{1}\,p_{2}\,\cdots\, p_{m+1}]} \left| f(p)-f(y) \right| $.\\ Since $f$ is uniformly continuous, $ \epsilon_{m} \rightarrow \, 0 \text{ as } m \rightarrow \infty$. Therefore, $\Delta\, (G_{\mu}f) = -f$.
\end{proof}
\bigskip 

\noindent
Based on the ideas developed in this paper, we now restate our main theorem \eqref{maintheorem} and prove the same to conclude this work. 
\bigskip 

\noindent 
\begin{theorem}
For any $f \in \mathcal{C}(\Sigma_N^+) \text{ and } \zeta \in \ell(V_0)$,  there exists a continuous function $u \in D_{\mu}$ such that the following holds:
\begin{align} \label{Drchlet}
\Delta u &= f, \nonumber \\
u|_{V_{0}} &= \zeta.
\end{align}
This solution is unique upto the harmonic functions taking value $0$ on the boundary $V_{0}$.
\end{theorem}
\begin{proof}
Define a function $u : \Sigma_{N}^{+} \longrightarrow \mathbb{R} $ as,
\[u \ \ := \ \ \sum\limits_{p\in V_{0}} \zeta(p)\, \chi_{p}^{m} - G_{\mu} f. \] 
Observe that $ \sum\limits_{p\in V_{0}} \zeta(p) \chi_{p}^{m}$ is a harmonic function and thus its Laplacian is $0$. Due to lemma \eqref{deltaG_muf=-f} and the linearity of the Laplacian $\Delta$ , we obtain $ G_{\mu}f \in D_{\mu}$ and thus $u \in D_{\mu} $ with
\[ \Delta u \ \ =\ \ -\,\Delta (G_{\mu} f) \ \ = \ \ f.\]
Since $(G_{\mu} f)|_{V_{0}}\ =\ 0$, clearly this choice of $u$ satisfies $u|_{V_{0}} = \zeta$. Further, if $h : \Sigma_{N}^{+} \longrightarrow \mathbb{R}$ is any harmonic function satisfying $ h|_{V_{0}} = 0 $, then it is trivial to see that the function $u+h$ is a solution to \eqref{Drchlet}.
\end{proof}
\bigskip 
\bigskip
\bigskip

\bibliographystyle{plainnat}

\bigskip 
\bigskip
\bigskip

\noindent 
\textsc{Shrihari Sridharan} \\ 
Indian Institute of Science Education and Research Thiruvananthapuram (IISER-TVM), \\ 
Maruthamala P.O., Vithura, Thiruvananthapuram, INDIA. PIN 695 551. \\
{\tt shrihari@iisertvm.ac.in}  
\bigskip \\ 

\noindent 
\textsc{Sharvari Neetin Tikekar} \\ 
Indian Institute of Science Education and Research Thiruvananthapuram (IISER-TVM), \\ 
Maruthamala P.O., Vithura, Thiruvananthapuram, INDIA. PIN 695 551. \\
{\tt sharvai.tikekar14@iisertvm.ac.in}

\end{document}